\documentclass[11pt]{amsart}
\usepackage{amssymb}
\usepackage{amsfonts}
\usepackage{amscd}
\usepackage{amsmath}
\usepackage{mathrsfs}
\usepackage{curves}
\usepackage{epsfig}
\usepackage[pass]{geometry}
\usepackage{graphicx}
\usepackage{verbatim}
\usepackage{latexsym}
\usepackage{eucal}
\usepackage{hyperref}
\usepackage{color}
\usepackage{dsfont}

\usepackage{soul}
\usepackage{cancel}

\numberwithin{equation}{section}

\newtheorem{theorem}{Theorem}[section]
\newtheorem{lemma}[theorem]{Lemma}
\newtheorem{prop}[theorem]{Proposition}
\newtheorem{cor}[theorem]{Corollary}

\newtheorem{definition}[theorem]{Definition}

\theoremstyle{remark}
\newtheorem{remark}[theorem]{Remark}

\def \supp {\text{\rm supp\hspace{.1em}}}
\def \Lip {\text{\rm Lip\hspace{.1em}}}

\def \beqq {\begin{equation}}
\def \eeqq {\end{equation}}

\def \bpf {\begin{proof}}
\def \epf {\end{proof}}
\def \beq {\begin{equation*}}
\def \eeq {\end{equation*}}

\def \eps {\epsilon}

\def \tilde {\widetilde}

\def \RR {{\mathbb R}}
\def \SS {{\mathbb S}}
\def \ii {\text{\rm i}}
\def \bH {\text{\bf H}}
\def \bfU {\text{\bf U}}
\def \bfu {\text{\bf u}}
\def \bfF {\text{\bf F}}
\def \bfG {\text{\bf G}}

\newcounter{daggerfootnote}

\title[Pencil-Beam approximation]{Pencil-Beam approximation of stationary Fokker-Planck}
\author{Guillaume Bal$^\dagger$}
\thanks{$^\dagger$Departments of Statistics and Mathematics, University of Chicago, Chicago, IL.\\ {\em Email address:} \texttt{guillaumebal@uchicago.edu}}
\author{Benjamin Palacios$^\ddagger$}
\thanks{$^\ddagger$Department of Statistics, University of Chicago, Chicago, IL.
\\ {\em Email address:} \texttt{bpalacios@uchicago.edu}}

\begin{document}
\maketitle

\begin{abstract}
Solutions of stationary Fokker-Planck equations in the narrow beam regime are commonly approximated by either ballistic linear transport or by a Fermi pencil-beam equation. We present a rigorous approximation analysis of these three models in a half-space geometry.  Error estimates are obtained in a 1-Wasserstein sense, which is an adapted metric to quantify beam spreading. The required well-posedness and regularity results for the stationary Fokker-Planck equation with singular internal and boundary sources are also presented in detail.
\end{abstract}

\section{Introduction}
We consider the stationary {\em(forward) Fokker-Planck equation} in the half-space $\RR^n_+=\{x=(x',x^n)\in\RR^n;x^n>0\}$ in dimension $n\geq 2$ given by:
 \begin{equation}\label{FP}
- \epsilon^2\sigma\Delta_\theta u  + \theta \cdot\nabla_x u + \lambda u= f,\quad (x,\theta )\in Q:=\RR^n_+\times\SS^{n-1},
\end{equation}
with $\nabla_x $ the spatial gradient and $\Delta_\theta$ the Laplace-Beltrami operator on the unit sphere $\SS^{n-1}$. The coefficients $\sigma(x)$ and $\lambda(x)$ are spatially-dependent functions bounded above and below by positive constants and $\epsilon>0$ is a scaling parameter. More concretely,
\begin{equation*}
0<\sigma_0\leq \sigma(x) \quad \text{and} \quad 0<\lambda_0 \leq \lambda,
\end{equation*}
for fixed constants $\sigma_0,\lambda_0$.

The Fokker-Planck equation \eqref{FP} may be used to model at a macroscopic level the propagation of high-frequency waves in heterogeneous media in the regime of highly peaked-forward scattering; see, e.g., \cite{BKR}. It may also be formally derived from a linear Boltzmann equation, also in the highly peaked forward scattering regime \cite{Po,A}. In \eqref{FP}, $f(x,\theta)$, $\epsilon^2\sigma(x)$ and $\lambda(x)$, represent an internal source/sink of particles at a point $x$ and direction $\theta$, the level of diffusion, and the amount of absorption,  respectively.

Our aim is to analyze the regime of narrow beam propagation. 
Heuristically, a beam propagating with a diffusion coefficient of order $O(1)$ changes directions over comparable distances or times with normalized speed. The narrow beam structure is therefore preserved only over short distances, which after appropriate rescaling takes the form  \eqref{FP} with a small diffusion coefficients of order $O(\eps^2)$.

The narrow beam will be initiated at the boundary of the domain, which for simplicity we assume to be a half space in most of the paper. 

In a general smooth domain $\Omega\subset\RR^n$, the boundary source of particles generates the following incoming boundary conditions for \eqref{FP}:
 \begin{equation}\label{FP_bc}
 u = g\quad\text{on}\quad \Gamma_- := \{(x,\theta)\in \partial\Omega\times\SS^{n-1} : \; \theta\cdot \nu(x)<0\},
 \end{equation}
where $\nu(x)$ is the outer unit normal vector at $x$ on the boundary. Similarly, the outgoing boundary is defined as
\begin{equation*}
\Gamma_+ := \{(x,\theta)\in\partial\Omega\times\SS^{n-1} : \;\theta\cdot \nu(x)>0\}.
\end{equation*}

Our subsequent analysis will repeatedly use solutions to the  {\em backward Fokker-Planck equation}, which is obtained by changing the sign of the transport operator $T=\theta\cdot\nabla_x$ in \eqref{FP} and imposing boundary conditions on the set $\Gamma_+$ instead.

We denote by $P$ the  Fokker-Planck operator and by $P^t$ its adjoint. The parameter $\epsilon>0$ measures the strength of the rescaled diffusion coefficient. We are interested in  the asymptotic regime $\epsilon\ll 1$ and $\sigma = O(1)$, where we use the standard notation $O(\epsilon^{m})$ whenever a term is bounded from above by a constant times the factor $\epsilon^m$, $m\geq 0$, and write $O(\epsilon^\infty)$ when the latter holds for every $m>0$. We use the symbol  $C$ as a generic positive constant depending on the dimension and coefficients of the problems but independent of $\epsilon$, and which may vary from line to line. 

There is a large literature on the analysis of the Fokker-Planck and similar kinetic equations. 
Existence and uniqueness questions are addressed in, e.g., \cite{DM-G,C,SH} and \cite{HJJ, HJJ2} from different points for view.
The question of regularity of solutions can be traced back to the classical works of hypoelliptic operators \cite{H} and subelliptic estimates \cite{K}. A more specific treatment of regularity in Fokker-Planck, based on sub-elliptic estimates, is presented in \cite{Bo}. A different approach to regularity can be found for instance in \cite{GIMV,WZ}, where iterative methods of the form of Moser or De Giorgi techniques are employed to prove H\"older continuity of solutions under sharper conditions on the coefficients. The results obtained from both strategies, subelliptic estimates (to obtain gains in regularity) and iterative methods (to gain continuity conditions), will be used in the derivations of this paper. 
There has been an increased interest in recent years in the study of the non-local Fokker-Planck models with fractional angular laplacians. Some references are \cite{AlS, GPR,IS}. Most of the analysis presented in this work should extend to the fractional case although we do not consider this issue in detail here. Let us also briefly mention recent work on the reconstruction of coefficients in the forward-peaked regime of Fokker-Planck in \cite{CLW2}.\\

The regime of small diffusions $\epsilon\ll 1$ in \eqref{FP}  appears, for instance, in the modeling of laser light propagation in turbulent atmospheres \cite{RoRe,HaBDH,KiKe,CCOCPH}. In the small diffusion scaling, one may be tempted to completely remove the diffusive part, which yields the following ballistic linear transport equation,
\begin{equation}\label{eq:BT}
\theta\cdot\nabla_x v + \lambda v = f,\quad (x,\theta)\in\RR^n_+\times\SS^{n-1},
\end{equation}
whose solutions can be explicitly written in terms of the source functions. 

We will consider the accuracy of this simple model. A more interesting and more accurate asymptotic description is given by the Fermi pencil-beam equation,
\begin{equation}\label{eq:FPB}
- \tilde{\sigma}\Delta_V U  + V \cdot\nabla_{X'} U+ \partial_{X^n}U  + \tilde{\lambda}U= f,\quad (X,V)\in \RR^{n}_+\times\RR^{n-1},
\end{equation}
with $\Delta_V$ the Euclidean Laplacian. The above equation may be formally obtained from \eqref{FP} (rescaled so that $\tilde{\sigma}$ is $O(1)$) by replacing the Laplace-Beltrami operator $\Delta_\theta$ by its approximation on the tangent plane $\Delta_V$ and assuming that the beam direction along $X^n$ is constant. Such solutions, unlike ballistic transport, capture beam spreading while completely neglecting backscattering (back to the boundary of $\RR^n_+$). They admit reasonably explicit expressions \cite{E}. For some derivations of the Fermi pencil-beam equation we refer to \cite{BL1}, while \cite{BL2} formally investigates the accuracy of this model with respect to Fokker-Planck and linear transport.

To analyze the accuracy of the model, it is convenient to link the Fermi pencil-beam and Fokker-Planck equation through a precise choice of local coordinates on the unit sphere and a rescaling of variables at the level of the diffusion, thus depending on $\epsilon$. This is done by means of {\em stereographic coordinates} (see \cite[p.35]{Lee}) from the south pole $S=(0,\dots,0,-1)$, given by the map $\mathcal{S}:\SS^{n-1}\backslash\{S\}\to\RR^{n-1}$ such that
\begin{equation}\label{def:sproj}
\mathcal{S}(\theta) := \Big(\frac{\theta_1}{1+\theta_n},\dots,\frac{\theta_{n-1}}{1+\theta_n}\Big),
\end{equation}
and with inverse $\mathcal{J} := \mathcal{S}^{-1}:\RR^{n-1}\to \SS^{n-1}\backslash\{S\}$,
\begin{equation}\label{def:isproj}
\mathcal{J}(v) = \Big(\frac{2v}{\langle v\rangle^2},\frac{1-|v|^2}{\langle v\rangle^2}\Big),\quad \langle v\rangle := (1+|v|^2)^{1/2}.
\end{equation}
Under these coordinates the local representation of the metric on the sphere is
\begin{equation*}
\mathring{g}=c^2(v)dx^2,\quad\text{with}\quad c(v):= \frac{2}{\langle v\rangle ^2},
\end{equation*}
while the volume form $d\theta$, the gradient $\nabla_\theta$ and the Laplace-Beltrami operator are respectively given by
\begin{equation*}
c^{n-1}(v)dv,\quad c^{-2}(v)\nabla_v,\quad\text{and}\quad c^{-(n-1)}(v)\nabla_v\cdot c^{n-3}(v)\nabla_v.
\end{equation*}
The coupling of stereographic coordinates with an $\epsilon$-rescaling of the transversal and angular variables defines the {\em stretched coordinate system}. Namely, given macroscopic coordinates $(x',x^n,\theta)$ in $\RR^n_+\times\SS^{n-1}$ we denote by 
\begin{equation*}
(X',X^n,V)\in \RR^{n-1}\times\RR_+\times\RR^{n-1}
\end{equation*}
the stretched local coordinates defined by the relation
\begin{equation*}
(x',x^n,\theta) = (2\epsilon X',X^n,\mathcal{J}(\epsilon V)).
\end{equation*}
This map is a diffeomorphism between $\RR^n_+\times(\SS^{n-1}\backslash \{S\})$ and $\RR^n_+\times\RR^{n-1}=:\mathcal{Q}$. 

Given the solution $u$ to \eqref{FP}-\eqref{FP_bc} with sources $f=0$ and formally $g = \delta(x)\delta(\theta-\eta)$, for $|\eta-N|=O(\epsilon^2)$ (here $N:=(0,\dots,0,1)$), its {\em pencil-beam approximation} takes the form
\begin{equation}\label{eq:FPBsol}
\mathfrak{u}(x,\theta) := (2\epsilon)^{-2(n-1)}U((2\epsilon)^{-1}x',x^n,\epsilon^{-1}\mathcal{S}(\theta)),
\end{equation}
for $U$ solution to the Fermi pencil-beam equation \eqref{eq:FPB} with null interior source, boundary condition $G(X',V)=\delta(X')\delta(V-\epsilon^{-1}\mathcal{S}(\eta))$, and coefficients
\begin{equation*}
\tilde{\sigma}= \frac{1}{4}\sigma(X^n\eta)\quad\text{and}\quad \tilde{\lambda}=\lambda( X^n\eta).
\end{equation*}
A similar approximation can be constructed by superposing pencil-beams if instead $g$ is a  nonnegative and integrable function on $\Gamma_-$. This will be considered later in the paper.\\


Our analysis provides a rigorous comparison of solutions of Fokker-Planck, ballistic linear transport and the pencil-beam approximation.  We establish error estimates in terms of powers of $\epsilon$. Since the solutions model narrow beams and are therefore quite singular, we need a metric that accounts for mass transport at the $\epsilon-$scale. We therefore use a natural setting for comparisons in such instances that is based on a version of the 1-Wasserstein distance.  

For a given relative open and bounded set $\Omega\subset\bar{\RR}^n_+$, we define
\begin{equation*}
BL_{1,\kappa}(\Omega\times\SS^{n-1}):=\{\psi\in C(\bar{\Omega}\times\SS^{n-1}): \|\psi\|_{\infty}\leq 1,\; \Lip(\psi)\leq\kappa\}.
\end{equation*}
The 1-Wasserstein distance between two finite Borel measures in $\Omega\times\SS^{n-1}$, $u(x,\theta)$ and $v(x,\theta)$, is defined as
\begin{equation}\label{eq:WD}
\mathcal{W}^1_{\kappa,\Omega}(u,v):=\sup\left\{\int \psi(x,\theta) (u-v): \psi\in BL_{1,\kappa}(\Omega\times\SS^{n-1})\right\},
\end{equation}
and analogously, if $u$ and $v$ are finite Borel measures defined in the whole space $\bar{\RR}^n_+\times\SS^{n-1}$, we define their 1-Wasserstein distance by replacing $\Omega$ with $\RR^n_+$ above and use the notation $\mathcal{W}^1_{\kappa}=\mathcal{W}^1_{\kappa,\RR^n_+}$. From the finiteness of the measures it is clear that
\begin{equation*}
\lim_{\Omega\to \RR^n_+}\mathcal{W}^1_{\kappa,\Omega}(u,v)=\mathcal{W}^1_{\kappa}(u,v).
\end{equation*}

The same distance was used in \cite{BJ} in the context of inverse transport theory to model possible mis-alignments in the sources used in the probing as well as detector blurring of singular solutions. Consider two points $a\neq b$, two distributions $\delta_a(x)=\delta(x-a)$ and $\delta_b(x)=\delta(x-b)$, an approximation to the identity $\varphi^\epsilon$, and regularizations in $L^1$, $\varphi^\epsilon_a(x) = \varphi^\epsilon(x-a)$ and $\varphi^\epsilon_b(x) = \varphi^\epsilon(x-b)$. For small enough $\epsilon$, $\|\varphi^\epsilon_a-\varphi^\epsilon_b\|_{L^1}=O(1)$ independent of the distance between $a$ and $b$ while $\mathcal{W}^1_\kappa(\delta_a,\varphi^\epsilon_a) = O(\epsilon\kappa)$ and $\mathcal{W}^1_\kappa(\delta_a,\delta_b)=\kappa|a-b|$ when $a$ and $b$ are close enough. Here, $\kappa$ may be interpreted as a confidence level in our measurements to distinguish between points $a$ and $b$, as well as the resolution we use to differentiate between measures. For the same reasons, this distance is a natural tool to quantify the spreading of narrowly focused beams. 
\\

The approximation by ballistic transport and Fermi pencil beams is interesting only in the setting of narrow beams. The $\epsilon-$broadening of an already ``broad" solution is hardly observable. We thus need to consider solutions with measure-valued boundary conditions. This poses no difficulty for the ballistic transport and Fermi pencil beam equations as those are explicitly solved by means of, e.g., Fourier transforms. We develop in this paper the appropriate theory for the Fokker-Planck equation to handle singular boundary conditions in the form of Radon measures. 
We obtain a unique solution $u$ to \eqref{FP}-\eqref{FP_bc} with internal and boundary sources given by finite compactly supported measures. In the particular cases of $f=0$ and $g(x,\theta)=\delta_{x_0}(x)\delta_{\eta}(\theta)$, for $|\eta+\nu(x_0)|=O(\epsilon^2)$, or $g(x,\theta)\in L^1(\partial\RR^n_+\times\SS^{n-1})$ such that $g\geq 0$ and $\|g\|_{L^1}=1$, compactly supported around $(x_0,-\nu(x_0))$ inside $\partial\RR^n_+\times\{|\theta+\nu(x_0)|<C\epsilon^2\}$, we find a second order approximation $\mathfrak{u}(x,\theta)$ by superposing delta-generated pencil-beams of the form \eqref{eq:FPBsol}. 
The corresponding ballistic transport solution of \eqref{eq:BT} is denoted by $v(x,\theta)$. 
In this context, the main result of this paper reads as follows. 
\begin{theorem}\label{thm:error_est_n3}
For dimension $n\geq 2$, there exist constants $C,C'>0$ (depending on $n$, $\sigma$ and $\lambda$) such that for any $\kappa\gtrsim 1$,
\begin{equation*}
\mathcal{W}^1_{\kappa}(u,v)\leq C\kappa\epsilon\quad\text{and}\quad \mathcal{W}^1_{\kappa}(u,\mathfrak{u})\leq C'\epsilon^2\kappa.
\end{equation*}
In addition, if the resolution parameter is taken such that $\kappa \approx \epsilon^{-1}$, then 
\begin{equation*}
\frac{1}{C}\leq \mathcal{W}^1_{\kappa}(u,v)\leq C,\quad\text{and}\quad \mathcal{W}^1_{\kappa}(u,\mathfrak{u})\leq C'\epsilon.
\end{equation*}
\end{theorem}
The last statement reflects the reasonable fact that at the $\epsilon-$scale of the beam, the ballistic transport solution $v$ does not accurately model the beam spreading. The Fermi pencil beam approximation, however, is still quite accurate for small $\epsilon$. \\

The rest of the paper is structured as follows. Section \ref{sec:L2theory} studies the well-posedness of the forward and backward Fokker-Planck equation in the $L^2$-setting based on a representation theorem that goes back to the work \cite{L}. Continuity of solution with respect to the data is given in the $L^2$ and $L^\infty$ settings. We then consider H\"older and Sobolev regularity, which are obtained by means of a local normal form of \eqref{FP}. The former regularity is a direct consequence of recent developments in the theory of ultra-parabolic equations, while the latter follows from a well-known commutator technique introduced in the work on subelliptic estimates \cite{K}. These results are collected in Section \ref{sec:rough_data} to conclude, following a duality argument, on the existence and uniqueness result for \eqref{FP} for rough sources and continuous coefficients. All necessary results on the the fundamental solution of the Fermi pencil-beam equation are presented in Section \ref{sec:Fermi}. Section  \ref{sec:1W_analysis} performs the approximation analysis using the 1-Wasserstein distance and states the main approximation result of this work, Theorem \ref{thm:error_est_n3}.

\section{Wellposedness and regularity of time-independent Fokker-Planck}\label{sec:L2theory}
We first establish existence and uniqueness of weak solutions based on a representation theorem for Hilbert spaces introduced in \cite{L}. Continuity with respect to the boundary and interior sources is also proven in the $L^2$-setting for general square integrable functions $f$ and $g$, while for vanishing boundary sources we also prove continuity with respect to $f$ in the $L^\infty$-norm. We then move to the question of regularity of solutions for which we introduce a (local) normal form for equation \eqref{FP}, and follow two different paths. As a direct application of recent developments on the study of Kolmogorov-type equations and more generally ultra-parabolic equations, we first deduce H\"older continuity of solutions under very weak assumptions on the coefficients. We then obtain an interior regularity result in Sobolev spaces by applying a well-known commutator technique that yields the necessary subelliptic estimates. Before going into the analysis a few remarks need to be taken into consideration regarding the results of this section.

\begin{remark}
(1) Everything we present in this section also holds for the backward Fokker-Planck equation and the proofs are identical but for minor differences such as interchanging the incoming boundary $\Gamma_-$ with the outgoing one $\Gamma_+$.
(2) In all the results presented in this section, the domain $\RR^n_+\times\SS^{n-1}$ may be replaced by any set of the form $\Omega\times\SS^{n-1}$, for $\Omega$ a bounded domain with sufficiently regular boundary. (3) The assumption of $\lambda_0>0$ is assumed to simplify the analysis and obtain global and integrable solutions on the whole half space.  Such an assumption would not be necessary on a bounded domain as long as Poincar\'e-type estimates are provided, something we do not consider here.
\end{remark}

We define $Q:= \RR^n_+\times\SS^{n-1}$ and write $L^{k}(\Gamma_{\pm})$, $k=1,2$, to denote the space of $L^{k}$ functions at the boundary sets $\Gamma_{\pm}$ endowed with the measure $|\theta\cdot\nu(x)|dS(x)d\theta$. Here, $d\theta $ stands for the volume form on $\SS^{n-1}$ while $dS(x)=dx'$ is the surface measure on $\partial\RR^n_+$. We also denote by $ L^2_{loc}(\bar{Q})$ the space of functions $u$ so that $\chi u\in L^2(\bar{Q})$ for any $\chi\in C^\infty_c(\RR^n)$.\\

\subsection{$L^2$-global theory and other properties} 
Consider the following Hilbert space
\begin{equation*}
\mathcal{H} := \{u\in L^2(Q): \nabla_\theta u\in L^2(Q)\}
\end{equation*}
with norm $\|u\|_{\mathcal{H}}^2 := \|u\|_{L^2}^2 + \epsilon^2\|\nabla_\theta u\|_{L^2}^2$ and $\mathcal{H}'$ its dual space. We denote by $T$ the transport operator $Tu = \theta \cdot\nabla_xu$. We look for solutions in the sub- Hilbert space 
\begin{equation*}
\mathcal{Y} := \{u\in \mathcal{H}:Tu\in \mathcal{H}'\} \text{ with norm }\|u\|_{\mathcal{Y}}^2:= \|u\|_{\mathcal{H}}^2 + \|Tu\|_{\mathcal{H}'}^2.
\end{equation*}

\begin{definition}\label{def:weak_sol}
A function $u\in\mathcal{Y}$ is a weak solution of \eqref{FP}-\eqref{FP_bc} if for all $\varphi\in C^\infty(\bar{Q})$ such that $\varphi = 0$ on $\Gamma_+$, it satisfies
\begin{equation}\label{weak_sol}
\begin{aligned}
&\int_{Q} \big(\epsilon^2\sigma(x)\nabla_\theta  u\cdot\nabla_\theta \varphi-u(\theta \cdot\nabla_x\varphi) + \lambda u\varphi \big)dxd\theta \\ 
&\hspace{8em}= \int_{Q}f\varphi dxd\theta + \int_{\Gamma_-}g\varphi |\theta\cdot\nu(x)|dS(x)d\theta.
\end{aligned}
\end{equation}
\end{definition}

\begin{theorem}[Wellposedness]\label{thm:wellposednessFP}
For any $f\in L^2(Q)$ and $g\in L^2(\Gamma_-)$, there exists a unique weak solution $u\in\mathcal{Y}$ to \eqref{FP}-\eqref{FP_bc}, which for some constant $C=C(\|\sigma\|_\infty,\sigma_0,\lambda_0)>0$ satisfies
\begin{equation*}
\|u\|_{\mathcal{Y}}\leq C\big(\|f\|_{L^2}+\|g\|_{L^2}\big).
\end{equation*}
Furthermore, the traces of $u$ on $\Gamma_-$ and $\Gamma_+$ are well defined, with $u|_{\Gamma_-} = g$ and $u|_{\Gamma_+}\in L^2(\Gamma_+)$. 
\end{theorem}

The main ingredient in the proof of existence is given by the next theorem. The rest of the proof follows the approach of \cite{DM-G} (see also \cite{C}).
\begin{theorem}[Lions \cite{L}]\label{thm:Lions} Let $\mathcal{H}$ be a Hilbert space provided with the inner product $(\cdot,\cdot)_{\mathcal{H}}$ and the norm $\|\cdot\|_{\mathcal{H}}$. Let $\mathcal{F}\subset\mathcal{H}$ be a subspace provided with a prehilbertian norm $|\cdot|$, such that the injection of $\mathcal{F}$ into $\mathcal{H}$ is continuous. Let us consider a bilinear form $a:\mathcal{H}\times\mathcal{F}\to\RR$, such that $a(\cdot,\varphi)$ is continuous on $\mathcal{H}$ for any $\varphi\in\mathcal{F}$ and such that $a(\varphi,\varphi)\geq \alpha|\varphi|^2$ for any $\varphi\in\mathcal{F}$ with $\alpha>0$. Then, given a linear form $L\in \mathcal{F}'$ continuous with the norm $|\cdot|$, there exists a solution $u\in\mathcal{H}$ of the problem: $a(u,\varphi)=L(\varphi)$ for any $\varphi\in\mathcal{F}$.
\end{theorem}

\begin{proof}[Proof of Theorem \ref{thm:wellposednessFP}]\quad\\
\noindent{\em 1) Existence.} 
Denote by $\mathcal{F}$ the set of test functions $\varphi\in C^\infty(\bar{Q})$ such that $\varphi = 0$ on $\Gamma_+$, and consider the prehilbertian norm on $\mathcal{F}$,
\begin{equation*}
|\varphi|_\mathcal{F}^2 := \|\varphi\|^2_\mathcal{H}+\frac{1}{2}\|\varphi\|^2_{L^2(\Gamma_-; |\theta\cdot\nu(x)|dS(x)d\theta )}.
\end{equation*}
It is clear that $\mathcal{F}$ is a subspace of $\mathcal{H}$ whose inclusion is continuous. In order to apply Theorem \ref{thm:Lions}, we define the bilinear form $a:\mathcal{H}\times \mathcal{F}\to \RR$ as the left-hand side of \eqref{weak_sol}:
\begin{equation*}
a(u,\varphi) := \int_{Q}\big(\epsilon^2 \sigma(x)\nabla_\theta  u\cdot\nabla_\theta \varphi-u(\theta \cdot\nabla_x\varphi)+ \lambda u\varphi \big)dxd\theta ,
\end{equation*}
and the bounded linear operator $L:\mathcal{F}\to\RR$ as the right-hand side of \eqref{weak_sol}:
\begin{equation*}
L(\varphi) = \int_{Q}f\varphi dxd\theta  + \int_{\Gamma_-}g\varphi |\theta\cdot\nu(x)|dS(x)d\theta. \end{equation*}
Under these notation, a weak solution to the FP equation should satisfies the equality
\begin{equation}\label{aL_eq}
a(u,\varphi)=L(\varphi),\quad\forall \varphi\in \mathcal{F}.
\end{equation}
For an arbitrary $\varphi\in \mathcal{F}$, $a(\cdot,\varphi)$ is a bounded linear operator in $\mathcal{H}$ and moreover $a(\cdot,\cdot)$ is coercive when restricted to $\mathcal{F}\times \mathcal{F}$. Indeed, for any $\varphi\in \mathcal{F}$ we have
\begin{equation*}
\begin{aligned}
a(\varphi,\varphi) &= \int_{Q}\big(\epsilon^2 \sigma(x)\nabla_\theta  \varphi\cdot\nabla_\theta \varphi +\lambda(x)\varphi^2\big)dxd\theta + \frac{1}{2}\int_{\Gamma_-}\varphi^2 |\theta\cdot\nu(x)|dS(x)d\theta \\
&\geq \min\{1,\sigma_0,\lambda_0\}|\varphi|^2_\mathcal{F}.
\end{aligned}
\end{equation*}
Here we used standard integration by parts in $\RR^n$ on the integral of $\varphi(\theta \cdot\nabla_x\varphi)$. It follows from Theorem \ref{thm:Lions} that there exists a weak solution $u\in \mathcal{H}$ satisfying \eqref{aL_eq}. In fact, $u\in\mathcal{Y}$ because $Tu$ is a distribution given by
\begin{equation}\label{trace_dist_def}
\langle Tu,\varphi\rangle =-\int_{Q} \big(\epsilon^2\sigma(x)\nabla_\theta  u\cdot\nabla_\theta \varphi + \lambda u\varphi  -f\varphi \big)dxd\theta  ,\quad \forall \varphi\in C^\infty_c(Q),
\end{equation}
thus $|\langle Tu,\varphi\rangle|\leq C\|\varphi\|_\mathcal{H}$. By density, the previous holds for all $\varphi\in \mathcal{H}$ therefore \eqref{trace_dist_def} defines $Tu\in \mathcal{H}'$ for solutions of \eqref{FP}. In addition we have:
\begin{equation}\label{ineq_opT}
\|Tu\|_{\mathcal{H}'}\leq\max\{\|\sigma\|_\infty,\|\lambda\|_\infty\}\|u\|_{\mathcal{H}}+ \|f\|_{L^2}.
\end{equation}

\noindent{\em 2) Trace.} To show uniqueness, we first need to make sense of the trace of a solution $u$ on $\Gamma_-$ (and $\Gamma_+$), and we do so by a density argument. The subset $\tilde{\mathcal{Y}}=C^\infty_c(\bar{Q}\backslash\Gamma_0)$ of $\mathcal{Y}$ is known to be dense \cite{Ba,C,DM-G}.
Take an arbitrary $\varphi\in\tilde{\mathcal{Y}}$ that vanishes on $\Gamma_+$. Then, Green's identity implies
\begin{equation*}
\|\varphi\|^2_{L^{2}(\Gamma_-;|\theta\cdot\nu(x)|dS(x)d\theta )} = -2\int_{Q}\varphi (T\varphi) dxd\theta \leq 2\|\varphi\|_\mathcal{H}\|T\varphi\|_{\mathcal{H}'}\leq \|\varphi\|^2_{\mathcal{Y}}.
\end{equation*}
Analogously, if $\varphi\in\tilde{\mathcal{Y}}$ vanishes on $\Gamma_-$ instead, we get
\begin{equation*}
\|\varphi\|^2_{L^{2}(\Gamma_+;|\theta\cdot\nu(x)|dS(x)d\theta )} \leq   2\|\varphi\|_\mathcal{H}\|T\varphi\|_{\mathcal{H}'}\leq \|\varphi\|^2_{\mathcal{Y}}.
\end{equation*}
Therefore, since any $\varphi\in\tilde{\mathcal{Y}}$ can be decomposed into $\varphi = \varphi_++\varphi_-$ with $\varphi_{\pm}$ vanishing on $\Gamma_{\pm}$, we deduce
\begin{equation*}
\|\varphi\|_{L^{2}(\Gamma_+\cup\Gamma_-;|\theta\cdot\nu(x)|dS(x)d\theta )} \leq \|\varphi\|_{\mathcal{Y}}.
\end{equation*}
Consequently, the density of $\tilde{\mathcal{Y}}$ implies $u|_{\Gamma_{\pm}}\in L^{2}(\Gamma_{\pm};|\theta\cdot\nu(x)|dS(x)d\theta )$ for all $u\in\mathcal{Y}$. 
In addition, for any pair $\varphi,\psi\in \tilde{\mathcal{Y}}$, 
we have
\begin{equation*}
\int_{Q}\varphi (T\psi) dxd\theta  + \int_{Q}\psi (T\varphi) dxd\theta  = \int_{\Gamma_-\cup\Gamma_+}\varphi\psi (\theta\cdot\nu(x))dS(x)d\theta .
\end{equation*}
By density we can extend the previous identity and write for any $u_1,u_2\in\mathcal{Y}$,
\begin{equation}\label{Greens_id}
\langle u_1,Tu_2\rangle_{\mathcal{H},\mathcal{H}'} + \langle u_2,Tu_1\rangle_{\mathcal{H},\mathcal{H}'} =\int_{\Gamma_-\cup\Gamma_+}u_1u_2(\theta\cdot\nu(x))dS(x)d\theta .
\end{equation}
In particular, if $u_1 = u$ is a weak solution of \eqref{FP} and we take $u_2 = \varphi\in\tilde{\mathcal{Y}}$ such that it vanishes on $\Gamma_+$, then we obtain
\begin{equation*}
0 = \int_{\Gamma_-}(u-g)\varphi |\theta\cdot\nu(x)|dS(x)d\theta,\quad\forall\varphi\in\tilde{\mathcal{Y}}.
\end{equation*}
This follows by recalling that $u$ satisfies \eqref{trace_dist_def} for all $\varphi\in \mathcal{H}$, thus in particular for all $\varphi\in\tilde{\mathcal{Y}}$ vanishing on $\Gamma_+$, and on the other hand, by definition of weak solution, for the same functions $\varphi$
\begin{equation}\label{weak_sol2}
\langle u,T\varphi\rangle_{\mathcal{H},\mathcal{H}'} = \int_{Q} \big(\epsilon^2\sigma(x)\nabla_\theta  u\cdot\nabla_\theta \varphi  + \lambda u\varphi -f\varphi \big)dxd\theta  - \int_{\Gamma_-}g\varphi |\theta\cdot\nu(x)|dS(x)d\theta .
\end{equation}

\noindent{\em 3) Uniqueness.} 
Let $w_1,w_2\in\mathcal{Y}$ be two weak solutions of \eqref{FP}-\eqref{FP_bc}. Therefore, $u=w_1-w_2$ is a solution of the same equation replacing $(f,g)$ by $(0,0)$. We first use the definition of $Tu$ in the distributional sense \eqref{trace_dist_def}, which also holds if we take $\varphi = u\in\mathcal{Y}\subset \mathcal{H}$. Therefore,
\begin{equation*}
\begin{aligned}
\langle u,Tu\rangle_{\mathcal{H},\mathcal{H}'}  &=-\int_{Q} \big(\epsilon^2\sigma(x)\nabla_\theta  u\cdot\nabla_\theta u  + \lambda u^2\big)dxd\theta. 
\end{aligned}
\end{equation*}
On the other hand, by plugging $u_1=u_2=u$ into \eqref{Greens_id} we obtain
\begin{equation*}
2 \langle u,Tu\rangle_{\mathcal{H},\mathcal{H}'}  = \int_{\Gamma_+}u^2 |\theta\cdot\nu(x)|dS(x)d\theta .
\end{equation*}
From the previous two equalities, one deduces
\begin{equation*}
\begin{aligned}
-\frac{1}{2}\int_{\Gamma_+}u^2 |\theta\cdot\nu(x)|dS(x)d\theta  =\int_{Q}\big( \epsilon^2\sigma(x)\nabla_\theta  u\cdot\nabla_\theta u + \lambda u^2\big)dxd\theta  \geq \min\{\sigma_0,\lambda_0\}\|u\|_\mathcal{H}^2, 
\end{aligned}
\end{equation*}
and consequently, since the left-hand side is nonpositive, the above is true only for $u=0$. 

\noindent{\em 4) Continuous dependence.} A weak solution $u$ of \eqref{FP}-\eqref{FP_bc} satisfies $Tu\in\mathcal{H}'$ given by  \eqref{trace_dist_def}. Choosing $\varphi=u$, we get
\begin{equation*}
\begin{aligned}
\langle u,Tu\rangle_{\mathcal{H},\mathcal{H}'}  &=-\int_{Q}\big( \epsilon^2\sigma(x)\nabla_\theta  u\cdot\nabla_\theta u  + \lambda u^2 - fu\big)dxd\theta. 
\end{aligned}
\end{equation*}
This together with  \eqref{Greens_id} applied to $u_1=u_2=u$ implies
\begin{equation}\label{cont_dep_1}
\min\{\sigma_0,\lambda_0\}\|u\|_\mathcal{H}^2\leq \int_{Q} fu dxd\theta + \frac{1}{2}\int_{\Gamma_-}g^2 |\theta\cdot\nu(x)|dS(x)d\theta,
\end{equation}
where the first integral in the right-hand side is bounded as follows
\begin{equation*}
\int_{Q} fu dxd\theta\leq \frac{1}{2}\delta^{-2}\|f\|^2_{L^2(\Omega\times\SS^{n-1})}+\frac{1}{2}\delta^{2}\|u\|^2_{L^2(\Omega\times\SS^{n-1})}.
\end{equation*}
We then choose $\delta>0$ small enough in order to absorb the last term with the left-hand side of \eqref{cont_dep_1}. Combining \eqref{ineq_opT} with the above inequalities concludes the proof.
\end{proof}

\begin{lemma}[Non-negative solutions]\label{lemma:min_pple}
For $f,g$ as in Theorem \ref{thm:wellposednessFP}, if $f,g\geq 0$ a.e., the unique solution to \eqref{FP}-\eqref{FP_bc} satisfies that $u\geq 0$ a.e. and the same holds for its trace on $\Gamma_-$. 
\end{lemma}
\begin{proof}
Let us denote $u^\pm = \max\{0,\pm u\}$ and set $A = Q\backslash\supp(u^+)$. 
We denote the characteristic function of $A$ as $\mathds{1}_A$. Then, $u^-\in \mathcal{H}$ with $\nabla_\theta u^- = (\nabla_\theta u)\mathds{1}_{A}$.  
Moreover, $T u^-\in \mathcal{H}'$ is a distribution given by
\begin{equation*}
\langle \phi,T u^-\rangle =-\int_A  \big(\epsilon^2\sigma(x)\nabla_\theta  u^-\cdot\nabla_\theta\phi +  \lambda(x)u^-\phi + f\phi \big)dxd\theta ,\quad\forall \phi\in C^\infty_c(A),
\end{equation*}
whose definition extends to all $\phi\in\mathcal{H}$ supported in $A$. On the other hand, identity \eqref{Greens_id} for $u^-$ gives
\begin{equation*}
2\langle  u^-,T u^-\rangle_{\mathcal{H},\mathcal{H}'} 
= \int_{\Gamma_-\cup\Gamma_+} |u^-|^2(\theta\cdot\nu(x))dS(x)d\theta.
\end{equation*}
From the previous two equalities and using that $u|_{\Gamma_-}=g\geq 0$, thus $u^-|_{\Gamma_-}=0$, we deduce that

\begin{equation*}
\begin{aligned}
&\int_A \lambda |u^-|^2dxd\theta + \frac{1}{2}\int_{\Gamma_+}|u^-|^2(\theta\cdot\nu(x))dS(x)d\theta\\
&\hspace{3em}= -\epsilon^2\int_A\sigma(x)|\nabla_\theta u^-|^2dxd\theta - \int_Af u^- dxd\theta, 
\end{aligned}
\end{equation*}
where by virtue of $f\geq 0$ it leads to
\begin{equation*}
\begin{aligned}
&\lambda_0\int_A| u^-|^2dxd\theta + \frac{1}{2}\int_{\Gamma_+}| u^-|^2(\theta\cdot\nu(x))dS(x)d\theta\leq 0\\
\end{aligned}
\end{equation*}
In consequence $ u^- = 0$. 
\end{proof}
Following \cite[Lemma 3.4]{HJJ}, we obtain the following continuity in the $L^\infty$-norm for vanishing boundary source. 

\begin{lemma}\label{lemma:Linfty_est}
If $f\in L^\infty(Q)$ compactly supported and $g=0$, then, the solution to \eqref{FP} satisfies the estimate $\|u\|_\infty\leq \lambda_0^{-1}\|f\|_\infty$.
\end{lemma}

\begin{proof}
Let $M=\|f\|_\infty$. 
Arguing by contradiction, let's assume there exists $\alpha>0$ and a bounded set $A\subset Q$ with positive measure such that, with out lost of generality, $u(x,\theta)>M\lambda_0^{-1}+\alpha$ on $A$. Notice that $u\in L^2(Q)$ since $f$ is also square integrable.

For any sufficiently small $\delta>0$ we can find a ball $B\subset Q$ such that 
\begin{equation}\label{meas_ineq}
\text{meas}(B\cap A)>(1-\delta)\text{meas}(B),
\end{equation}
(see for instance \cite{St}).
We take $h\in C^\infty_c(Q)$ so that $h\geq 0$, $\supp h\subset\bar{B}$ and
\begin{equation*}
\left|\int\Big(\frac{\chi_B}{\text{meas}(B)} - h\Big)dxd\theta\right| = \left|1-\int h dxd\theta\right|<\delta,
\end{equation*}
and consider $\varphi$, solution to the backward Fokker-Planck equation
\begin{equation}\label{b_FP}
 -\epsilon^2\sigma(x)\Delta_\theta \varphi- \theta\cdot\nabla_x \varphi + \lambda \varphi= h,\quad \varphi|_{\Gamma_+} = 0.
 \end{equation}
According to lemma \ref{lemma:min_pple} applied to \eqref{b_FP} we deduce $\varphi\geq 0$.
Then, we have
\begin{equation*}
\int_{Q} f \varphi dxd\theta \leq M\int_{Q} \varphi dxd\theta \leq M\lambda_0^{-1}(1+ \delta).
\end{equation*}
The last inequality follows from the fact that $\varphi$ is indeed a strong solution of \eqref{b_FP} (as we will see in \S\ref{sssec:reg}) and then integrating \eqref{b_FP} over $Q$. Indeed, 
\begin{equation*}
\int_{Q}  hdxd\theta = \int_Q\big(- \theta\cdot\nabla_x \varphi + \lambda \varphi\big)dxd\theta
\end{equation*}
where we used Green's identity on the unit sphere: $\int_{\SS^{n-1}}\Delta_\theta \varphi d\theta= 0$. Stoke's theorem and the lower bound $\lambda_0\leq \lambda$ then yield
\begin{equation*}
\begin{aligned}
\lambda _0\int_{Q}  \varphi dxd\theta &\leq\int_{Q}  hdxd\theta 
+\int_{\Gamma_-}\varphi (\theta\cdot\nu(x))dS(x)d\theta\leq \|h\|_{L^1},
\end{aligned}
\end{equation*}
where by definition of $h$, $\|h\|_{L^1}\leq 1+\delta$. 

On the other hand,
\begin{equation*}
\int_{Q} f \varphi dxd\theta  = \int_{Q} u h dxd\theta  = I_1 + I_2,
\end{equation*}
with 
\begin{equation*}
\begin{aligned}
I_1 &= \int_{B\cap A} uh dxd\theta \\
&\geq (M\lambda_0^{-1}+\alpha)\left(\int_{B\cap A}\frac{\chi_B}{\text{meas}(B)}dxd\theta - \int_{B\cap A}\left( \frac{\chi_B}{\text{meas}(B)}- h\right)dxd\theta\right)\\
&\geq (M\lambda_0^{-1}+\alpha)\left(\frac{\text{meas}(B\cap A)}{\text{meas}(B)} - \delta\right) >(M\lambda_0^{-1}+\alpha)(1-2\delta),
\end{aligned}
\end{equation*}
and
\begin{equation*}
\begin{aligned}
I_2 = \int_{B\backslash A} u h dxd\theta \leq \|u\|_{L^2}\|h\|_\infty\text{meas}(B\backslash A)^{1/2}
\leq \|u\|_{L^2}\|h\|_\infty\text{meas}(B)^{1/2}\delta^{1/2}.
\end{aligned}
\end{equation*}
Bringing the above together, we obtain that for some $C>0$ independent of $\delta$,
\begin{equation*}
(M\lambda_0^{-1}+\alpha)(1-2\delta) - C\delta^{1/2}\leq M\lambda_0^{-1}(1+\delta),
\end{equation*}
which then implies that
\begin{equation*}
M\lambda_0^{-1}+\alpha \leq M\lambda_0^{-1} + C\delta^{1/2}.
\end{equation*}
It remains to take $\delta$ small enough to obtain a contradiction.
\end{proof}

\subsection{Regularity}\label{subsec:regularity}

The regularity of solution to Fokker-Planck is obtained by using an appropriate local representation of the equation. These {\em beam coordinates} were already used in \cite{BL1} to derive the Fermi pencil-beam equation from an asymptotic expansions of Fokker-Planck and linear Boltzmann equations.

\subsubsection{Beam coordinates}
Consider the map $\mathcal{B}:\SS^{n-1}_+\to\RR^{n-1}$  transforming the upper hemisphere $\SS^{n-1}_+:=\{\theta\in\RR^{n-1}:|\theta|=1,\; \theta_n>0\}$ and given by
\begin{equation*}
\mathcal{B}(\theta):= \Big(\frac{\theta_1}{\theta_n},\dots,\frac{\theta_{n-1}}{\theta_n}\Big),
\end{equation*}
and whose inverse is defined as the map $\mathcal{B}^{-1}:\RR^{n-1}\to \SS^{n-1}_+$  such that
\begin{equation*}
\mathcal{B}^{-1}(v) =  \Big(\frac{v}{\langle v\rangle},\frac{1}{\langle v\rangle}\Big).
\end{equation*}
Under such coordinates, we obtain the local representations:
\begin{equation*}
\begin{aligned}
&\mathring{g}=\frac{1}{\langle v\rangle^2}\Big(Id - \Big(\frac{v}{\langle v\rangle}\Big)\Big(\frac{v}{\langle v\rangle}\Big)^T\Big),\quad d\theta= \langle v\rangle^{-n}dv,\\
&\nabla_\theta=\langle v\rangle^2(Id+vv^T)\nabla_v,\quad\text{and}\quad \Delta_\theta= \langle v\rangle^n\nabla_v\cdot \langle v\rangle^{2-n}(Id+vv^T)\nabla_v.
\end{aligned}
\end{equation*}

\subsubsection{Normal form and H\"older regularity}\label{subsec:normal_form}
The regularity properties for  Fokker-Planck  solutions are obtained by assigning to  \eqref{FP} a normal form for which regularity issues are well understood. 
The same analysis applies to backward Fokker-Planck.

Let us pick an arbitrary $\theta_0\in\SS^{n-1}$. By embedding the unit sphere in $\RR^n$, we consider spatial coordinates so that $\theta_0= (0,\dots,0,1)=:N$. On the unit sphere, we choose beam-coordinates mapping the hemisphere $\{\theta\in\SS^{n-1}: \theta\cdot\theta_0>0\}$ onto $\RR^{n-1}$. 

Let $U_x\subset\RR^n_+$ be a neighborhood of an arbitrary point $x_0\in \RR^n_+$ and $U_v$ a neighborhood of the origin in $\RR^{n-1}$. For any test function $\tilde{\phi} = \langle v\rangle^{n+1}\phi\in C^\infty_c(U_x\times U_v)$, in local coordinates, equation \eqref{FP} takes the form:
\begin{equation}\label{FP_loc_coord}
\begin{aligned}
\int -u\partial_{x^n}\phi-u v \cdot\nabla_{x'}\phi+\langle\frac{\epsilon^2\sigma(x)}{\langle v\rangle^{n-2}}(Id + vv^T)&\nabla_v u,\nabla_v(\langle v\rangle^{n+1}\phi)\rangle \\
& + (\langle v\rangle\lambda)u\phi \;dxdv = \int (\langle v\rangle f )\phi dxdv.
\end{aligned}
\end{equation}
We re-label coordinates by writing $(y,t):=(x',x^n)$, and define
\begin{equation}\label{opAB}
A(t,y,v)= \epsilon^2\sigma(y,t)\langle v\rangle^{3}(Id + vv^T),\quad B(t,y,v) = \epsilon^2(n+1)\sigma(y,t)\langle v\rangle^3 v;
\end{equation}
and $c(y,v) = \langle v\rangle \lambda(y,t)$, $\hat{f}(t,y,v) = \langle v\rangle f(t,y,v)$. Then, assuming enough regularity on $u(y,t,v)$ one verifies that it satisfies the equation
\begin{equation}\label{KFP_eq}
\partial_tu + v\cdot\nabla_{y}u= \nabla_v\cdot(A\nabla_v u) - B\cdot\nabla_v u - cu+\hat{f},\quad \text{in }U_x\times U_v,
\end{equation}
in the weak sense, i.e. $u,\nabla_vu,\partial_tu + v\cdot\nabla_{y}u\in L^2_{loc}(U_x\times U_v)$. 
We call this equation the {\em normal form} of \eqref{FP}. We point out that the matrix-valued function $A$ is positive semi-definite at every point while the vector-valued function $B$ and the scalar function $c$ are bounded in $U_x\times U_v$.

If a null boundary source is imposed, that is $g=0$, we can also write equation \eqref{FP} in normal form around (outgoing) boundary point $(x_0,\theta_0)\in\Gamma_+$ by extending the problem across $\Gamma_+$. Indeed, let $(x_0,\theta_0)\in \Gamma_+$ and $\mathcal{U}\subset\RR^{n}\times\SS^{n-1}$ a sufficiently small neighborhood of $(x_0,\theta_0)$. Denote by $u_{out}$ the unique solution to the  exterior problem
\begin{equation}\label{FP_ext1}
-\epsilon^2\sigma\Delta_\theta u_{out}+ \theta\cdot\nabla_xu_{out} + \lambda u_{out} = 0\quad\text{in}\quad \RR^n_-\times\SS^{n-1},
\end{equation}
with (exterior) incoming boundary conditions
\begin{equation*}
u_{out} = u \quad\text{ on }\Gamma_+.
\end{equation*}
Here, $\sigma$ and $\lambda$ are any positive extension of the coefficients that preserve the original regularity.
It turns out that $u_{ext} = u|_{\bar{\RR}^n_+}+u_{out}|_{\RR^n_-}$ satisfies that
\begin{equation*}
\int \epsilon^2\sigma\nabla_{\theta}u_{ext}\cdot\nabla_{\theta}\phi - u_{ext}(\theta\cdot\nabla_{x}\phi) + \lambda u_{ext}\phi \;dxd\theta = \int f_{ext}\phi \;dxd\theta,
\end{equation*}
for all $\phi\in C_c^\infty(\mathcal{U})$, with $f_{ext}$ the zero extension of $f$. We can choose beam coordinates with respect to $(x_0,\theta_0)$ and deduce that $u_{ext}$ satisfies (in a weak sense) an equation of the form \eqref{KFP_eq}, for all $(y,t,v)$ in a bounded open set.

A first consequence of the normal form \eqref{KFP_eq} is the H\"older continuity of solutions.
This issue has been studied by several authors in the context of equations of the form \eqref{KFP_eq} (see for instance the recent results \cite{GIMV,WZ}, and the survey \cite{Br}). 
The following is a direct consequence of Theorem 1.1 in \cite{WZ} (or [Theorem 3 in \cite{GIMV}] in the case $c=0$).

\begin{theorem}\label{thm:Holder_reg}
Let $u$ be solution to \eqref{FP} for $f,\sigma,\lambda\in L^\infty(\bar{Q})$. Then, there is $\alpha\in(0,1)$ so that $u\in C^{0,\alpha}(Q)$. If in addition $g=0$ in \eqref{FP}-\eqref{FP_bc} then $u\in C^{0,\alpha}(Q\cup\Gamma_+)$.
\end{theorem}

\begin{remark}
In the context of backward Fokker-Planck with null boundary condition $g=0$ on $\Gamma_+$, the analogous result gives $u\in C^{0,\alpha}(Q\cup\Gamma_-)$.
\end{remark}

\subsubsection{Sobolev regularity and strong solutions}\label{sssec:reg}
Let us introduce the operators $D^{s}_{x}:=(I-\Delta_{x})^{s/2}$, with $s\in\RR$, defined by
\begin{equation*}
D^{s}_{x}u= \frac{1}{(2\pi)^{n/2}}\int_{\RR^{n}}e^{\ii k\cdot x}(1+|k|^2)^{s/2}\hat{u}(k)dk,\quad\forall u\in C^\infty_c(\RR^{n}),
\end{equation*}
where $\hat{\;\;}$ stands for the Fourier transform. Following the notation $x=(x',x^n)$, we analogously define $D_{x'}$ for an operator acting only on $x'$.

\begin{theorem}\label{thm:higher_reg}
Let $u$ be a weak solution to \eqref{FP} with source $f\in L^2(Q)$. Then, 
\begin{equation*}
\Delta_\theta u,\; \theta\cdot\nabla_xu,\; D^{2/3}_xu\in L^2_{loc}(Q),
\end{equation*}
and in particular $u$ is a strong solution.
Moreover, for every compact $\mathcal{K}$ and open $\mathcal{O}$ such that $\mathcal{K}\subset\mathcal{O}\subset Q$, there exists a constant $C>0$ so that
\begin{equation}\label{subslliptic_est}
\begin{aligned}
&\epsilon^{2}\|\Delta_\theta u\|_{L^2(\mathcal{K})}+\epsilon^{2/3}\|D^{2/3}_xu\|_{L^2(\mathcal{K})}+\|\theta\cdot\nabla_xu\|_{L^2(\mathcal{K})}\\
&\hspace{5em} \leq C\big(\|f\|_{L^2(\mathcal{O})} + \epsilon^2\|\nabla_\theta u\|_{L^2(Q)} + \|u\|_{L^2(Q)}\big).
\end{aligned}
\end{equation}
\end{theorem}
\begin{remark}
Estimates of the form \eqref{subslliptic_est} are usually referred to as {\em subelliptic estimates}, in this case with a gain of 2 derivatives in the angular variable and  $2/3$ derivatives in the spatial variable. Assuming smoothness in $f$, $\sigma$ and $\lambda$, repeated differentiation of \eqref{FP} combined with the subelliptic estimates lead to the {\em hypoellipticity} property.
\end{remark}
\begin{remark}
According to \S\ref{subsec:normal_form}, extending the solution near a point $(x,\theta)\in\Gamma_+$ across  the boundary implies $\Delta_\theta u,\; \theta\cdot\nabla_xu,\; D^{2/3}_xu\in L^2(\mathcal{K})$ for any compact $\mathcal{K}\subset Q\cup\Gamma_+$.
\end{remark}

We give a brief proof of the theorem and we refer the reader to Appendix \ref{appdx:proofThm_reg} for a more
detailed one.

\begin{proof}
It is enough to obtain the estimates in a neighborhood $\mathcal{U}$ of an interior point $(x_0,\theta_0)\in Q$. Up to some rotation, we can always assume that $\theta_0$ is contained in the span of $(1,0\dots,0)$ and $(0,\dots,0,1)=N$. We first consider beam coordinates on the upper hemisphere $\SS^{n-1}_+=\{\theta^n>0\}$, thus $N=(0,\dots,0,1)$ is identified with $0\in \RR^{n-1}$.

After standard mollification and localization arguments we reduce the problem to obtaining the desired estimates for a smooth compactly supported function $u$, solution to the normal form equation
\begin{equation}\label{eq_T}
T u -\nabla_v\cdot(A\nabla_vu) + B\cdot\nabla_vu+ c u = f,\quad\forall (x,v)\in \mathcal{U},
\end{equation}
with $f$ (different from the original one) a smooth and compactly supported function with $L^2$-norm bounded by the norm of the original source and the $\mathcal{H}$-norm of the Fokker-Planck solution. Here $T$ stands for the transport operator
\begin{equation*}
T = (\partial_{x^n} + v\cdot\nabla_{x'}),
\end{equation*}
which satisfies the commutator identity: 
\begin{equation}\label{commutator_id}
\partial_{x^j} = \partial_{v_j}T - T\partial_{v^j},\quad j=1,\dots,n-1.
\end{equation}
The bulk of the proof is analogous to \cite{Bo} where a central role is played by the (H\"ormander-type) identity \eqref{commutator_id}. The main difference is the more general form of the principal term,  which is easily addressed due to the positive-definiteness of the anisotropic coefficient $A$. The technique consists in obtaining the estimates
\begin{equation}\label{sub_est1}
\|\nabla_v\cdot A\nabla_v u\|_{L^2}\leq C\|f\|_{L^2}\quad\text{and}\quad
\| D^{2/3}_{x'}u \|_{L^2}\leq C\epsilon^{-2/3}\|f\|_{L^2},
\end{equation}
from which one deduces
\begin{equation*}
\epsilon^{2}\|\Delta_\theta u\|_{L^2}\leq C\|f\|_{L^2}.
\end{equation*}
On the other hand, by considering instead beam coordinates on the hemisphere $\{\theta\in\SS^{n-1}:\theta^1>0\}$ we can repeat the computations leading to \eqref{sub_est1} but now in terms of the operator $D^s_{x''}= (1-\Delta_{x''})^{s/2}$, where we decompose the spatial variables as $x = (x^1,x'')$. We then obtain
\begin{equation*}
\| D^{2/3}_{x''}u \|_{L^2}\leq C\epsilon^{-2/3}\|f\|_{L^2},
\end{equation*}
which implies an analogous inequality for the full derivative $D^{2/3}_{x}u$.

The proof is complete by considering a convergent subsequence with respect to the mollification parameter, with respective derivatives converging weakly in $L^2$, in a smaller neighborhood $\mathcal{V}$ of $(x_0,\theta_0)$, to $\Delta_\theta u$ and $D^{2/3}_xu$ for $u$ now the Fokker-Planck solution.  The estimates for the mollification yield the following subelliptic estimate for the limit,
\begin{equation*}
\epsilon^{2}\|\Delta_\theta u\|_{L^2(\mathcal{V})} +\| D^{2/3}_xu \|_{L^2(\mathcal{V})} \leq C(\|f\|_{L^2(\mathcal{W})} + \|u\|_{\mathcal{H}}),
\end{equation*}
where here $f$ is the source term in \eqref{FP} and $\mathcal{W}\subset Q$ a slightly larger open set containing $\mathcal{U}$. Finally, it follows directly from the equation that
\begin{equation*}
\|\theta\cdot\nabla_x u\|_{L^2(\mathcal{V})}\leq C(\|f\|_{L^2(\mathcal{W})}+ \|u\|_{\mathcal{H}}).
\end{equation*}
\end{proof}

\section{Fokker-Planck with singular sources}\label{sec:rough_data}
Our main motivation to study Fokker-Planck equation is the modeling of narrow beams with a source represented as a delta distribution. The above well-posedness theory does not allow for such objects, and hence it is necessary to extend the equation to delta sources or more general finite compactly supported measures. We achieve this following a duality argument which implies the need of continuity results for solution to the backward Fokker-Planck equation. These results follow from \S\ref{sec:L2theory} and the paragraph at the beginning of that section. 

All the measures considered here will be of Borel type. Let $f$ and $g$ be finite measures in $Q$ and $\Gamma_-$ respectively. Below we denote by $P$ the forward Fokker-Planck operator while $P^t$ stands for the adjoint (or backward) Fokker-Planck operator.

\begin{definition}\label{def:dist_sol}
A distribution $u\in\mathcal{D}'(\bar{Q})$ is said to be a solution of \eqref{FP}-\eqref{FP_bc} whenever
\begin{equation*}
\begin{aligned}
\langle Pu,\varphi\rangle := \langle u,P^t\varphi\rangle = \langle f,\varphi \rangle + \langle g|\theta\cdot\nu(x)|,\varphi\rangle_{\Gamma_-},
\end{aligned}
\end{equation*}
for all $\varphi\in C^\infty_0(\bar{Q})$ such that $\varphi = 0$ on $\Gamma_+$.
\end{definition}
\begin{remark}\label{rmk:meas_sol}
(1) The factor $|\theta\cdot\nu(x)|$ is introduced to agree with the standard definition of weak solutions for Sobolev spaces in Definition \eqref{def:weak_sol}. (2) Notice that if $u$ is a finite measure, the above extends to all $\varphi$ solutions to \eqref{backFP} for some $\psi\in C_c(\bar{Q})$, the set of all continuous functions with compact support in $\bar{Q}$. This is because those solutions belong to $C(\bar{Q})\cap L^2(Q)$ and can be approximated (with respect to $\|\cdot\|_\infty$) by compactly supported smooth functions.
\end{remark}

\begin{theorem}\label{thm:FP_dist}
Let $\sigma,\lambda\in C(\bar{\RR}^n_+)$ and $f$ and $g$ (positive) finite measures with compact support, defined respectively on $Q$ and $\Gamma_-$. Then, there is a unique (positive) finite measure in $Q$ satisfying \eqref{FP}-\eqref{FP_bc}.
\end{theorem}

\begin{proof}
Due to the H\"older-regularity result of Theorem \ref{thm:Holder_reg} and particularly the remark after it, and Lemma \ref{lemma:Linfty_est}, the source-to-solution map associated to the backward equation
\begin{equation}\label{backFP}
 -\epsilon^2\sigma(x)\Delta_\theta \varphi- \theta\cdot\nabla_x \varphi + \lambda \varphi= \psi,\quad \varphi|_{\Gamma_+} = 0,
\end{equation}
maps $S_{adj}:C_c(\bar{Q})\mapsto C(Q\cup\Gamma_-)$ continuously. It is a well defined map due to the inclusion $C_c(\bar{Q})\subset L^2(Q)$. 

We claim that $u$, defined by
\begin{equation}\label{dist_sol}
\langle u,\psi\rangle := \langle f,S_{adj}\psi \rangle_{Q} + \langle g|\theta\cdot\nu(x)|,S_{adj}\psi\rangle_{\Gamma_-},\quad\forall \psi\in C_c(\bar{Q}),
\end{equation}
is a finite measure and satisfies \eqref{FP}-\eqref{FP_bc} in the distributional sense. It is clearly a bounded linear map on $C_c(\bar{Q})$, thus by (Radon-Riesz representation) duality it defines a Radon measure on $\bar{Q}$, and it is finite since $f$ and $g$ are. Moreover, since any $\varphi\in C^\infty_c(\bar{Q})$ satisfying that $\varphi|_{\Gamma_+}=0$ can be regarded as the unique solution to  the backward system for $\psi= P^t\varphi\in C^\infty_c(\bar{Q})$, then
\begin{equation*}
\langle f,\varphi\rangle_{Q} +\langle g|\theta\cdot\nu(x)|,\varphi \rangle_{\Gamma_-}= \langle u,P^t\varphi\rangle_Q ,
\end{equation*}
which means $u$ is a weak solution. 

Let $v$ be another finite measure solution to \eqref{FP}-\eqref{FP_bc}. For an arbitrary $\psi\in C_c(\bar{Q})$ and denoting $\varphi = S_{adj}\psi$ we have
\begin{equation*}
\langle u-v,\psi\rangle =  \langle f,\varphi \rangle + \langle g|\theta\cdot\nu(x)|,\varphi\rangle_{\Gamma_-} - \langle v,P^t\varphi\rangle=0,
\end{equation*}
which implies $u=v$.
The positivity is a direct consequence of the definition and the analogous result of Lemma \ref{lemma:min_pple} for the backward problem.
\end{proof}
\bigskip

\section{Fermi pencil-beam equation}\label{sec:Fermi}

This section collects all the results we need on the Fermi pencil-beam solutions.

For $\tilde{\sigma}$ and $\tilde{\lambda}$ continuous functions depending only on $X^n\geq0$ and bounded from below by positive constants, the Fermi pencil-beam equation is given by
\begin{equation}\label{FPb}
\mathcal{P}(U):=- \tilde{\sigma}\Delta_V U  + V \cdot\nabla_{X'} U+ \partial_{X^n}U  + \tilde{\lambda}U= F,\quad (X,V)\in \RR^{n}_+\times\RR^{n-1},
\end{equation}
and is endowed with boundary condition
\begin{equation}\label{FPb_bc}
U|_{X^n=0} = G(X',V),\quad (X',V)\in\RR^{2(n-1)}.
\end{equation}
The following solvability result goes back to work of Eyges \cite{E} on transport theory of charged particles. Below we write $\mathfrak{F}_{X'}[\cdot]$ to denote the Fourier transformation with respect to the variable $X'$, that is
\begin{equation*}
\mathfrak{F}_{X'}[F](\xi) := \frac{1}{(2\pi)^{(n-1)/2}}\int e^{-\ii X'\cdot\xi}F(X')dX'.
\end{equation*}
$\mathfrak{F}_{V}$ is given analogously for a phase variable $\eta$. We also write $\tau_{Y'}$ to denote the translation map on the spatial transversal variables, i.e. $[\tau_{Y'}f](X') := f(X'-Y')$.

\begin{prop}\label{prop:FBP_wellposedness}
For any $F\in C(\RR_+;\mathcal{S}'(\RR^{2(n-1)}))$ and $G\in\mathcal{S}'(\RR^{2(n-1)})$, there exists a unique solution 
$U\in C^1(\RR_+;\mathcal{S}'(\RR^{2(n-1)}))$ 
to \eqref{FPb}-\eqref{FPb_bc} given explicitly by 
\begin{equation}\label{FPb_sol_form}
\begin{aligned}
U(X,V)&=e^{-\int^{X^n}_0\tilde{\lambda}(s)ds}\tau_{X^nV}({\bf H}_1*G) \\
&\hspace{2em}+ \int^{X^n}_0e^{-\int^{X^n}_t\tilde{\lambda}(s)ds}\tau_{X^nV}({\bf H}_2(t)*\tau_{-tV}F)dt,
\end{aligned}
\end{equation}
for $\bH_1(X,V)$ and $\bH_2(X,V;t)$ Gaussian kernels, namely
\begin{equation}\label{def:H1}
{\bf H}_1(X,V):=
\frac{e^{- \frac{1}{4(E_2E_0-E_1^2)}\big(E_0|X'|^2 + 2E_1X'\cdot V + E_2|V|^2\big)}}{(4\pi\sqrt{E_2E_0-E_1^2})^{n-1}}
\end{equation}
with $E_k(X^n) := \int^{X^n}_0s^k\tilde{\sigma}(s)ds$, and $\bH_2$ defined analogously by replacing the terms $E_k$ with $\tilde{E}_k(X^n;t) := \int^{X^n}_ts^k\tilde{\sigma}(s)ds$.
\end{prop}
\begin{proof}
A solution $U$ to \eqref{FPb} should satisfy (formally)
\begin{equation*}
(\partial_{X^n}+\ii V\cdot\xi -\tilde{\sigma}(X^n)\Delta_V + \tilde{\lambda})\mathfrak{F}_{X'}[U]=\mathfrak{F}_{X'}[F].
\end{equation*}
Define 
\begin{equation*}
\begin{aligned}
&\bfU(\xi,X^n,V) = e^{ \ii X^n(V\cdot \xi) + \int^{X^n}_0\tilde{\lambda}(t)dt}\mathfrak{F}_{X'}[U]\\
&\text{and}\quad \bfF(\xi,X^n,V) = e^{ \ii X^n(V\cdot \xi) + \int^{X^n}_0\tilde{\lambda}(t)dt}\mathfrak{F}_{X'}[F],
\end{aligned}
\end{equation*}
for which the following equality holds
\begin{equation*}
\partial_{X^n}\bfU -\tilde{\sigma}(X^n)e^{ \ii X^n(V\cdot \xi) + \int^{X^n}_0\tilde{\lambda}(t)dt}\Delta_V\big(e^{- \ii X^n(V\cdot \xi) - \int^{X^n}_0\tilde{\lambda}(t)dt}\bfU\big)=\bfF.
\end{equation*}
Denoting now $\hat{\bfU} = \mathfrak{F}_V[\bfU](\xi,X^n,\eta)$ and $\hat{\bfF} = \mathfrak{F}_V[\bfF](\xi,X^n,\eta)$, and recalling the identity 
\begin{equation*}
\mathfrak{F}_V[e^{\ii X^n(V\cdot\xi)}h(V)](\eta) = \mathfrak{F}_V[h(V)](\eta-X^n\xi),
\end{equation*}
by applying FT with respect to $V$ in the previous equality, we obtain the following equation
\begin{equation*}
\partial_{X^n}\hat{\bfU} + \tilde{\sigma}(X^n)|\eta - X^n\xi|^2\hat{\bfU}=\hat{\bfF},
\end{equation*}
which we endow with the boundary condition $\hat{\bfU}(\xi,0,\eta) = \mathfrak{F}_{X',V}[G](\xi,\eta)=:\bfG(\xi,\eta)$. $\hat{\bfU}$ is then given by
\begin{equation*}
\hat{\bfU}(\xi,X^n,\eta) = \bfG(\xi,\eta)e^{-\int^{X^n}_0|\eta - t\xi|^2\tilde{\sigma}(t)dt} + \int^{X^n}_0e^{-\int^{X^n}_t|\eta - s\xi|^2\tilde{\sigma}(s)ds}\hat{\bfF}(\xi,t,\eta)dt,
\end{equation*}
with $\hat{\bfF}(\xi,t,\eta) = e^{\int^t_0\tilde{\lambda}(s)ds}\mathfrak{F}_{X',V}[F](\xi,t,\eta - t\xi)$. Consequently, the FT of $U$ has the form
\begin{equation*}
\begin{aligned}
\mathfrak{F}_{X',V}[U](\xi,X^n,\eta) &= e^{-\int^{X^n}_0\tilde{\lambda}(t)dt}\hat{\bfU}(\xi,X^n,\eta + X^n\xi).
\end{aligned}
\end{equation*}
The explicit expression for $U$ follows directly from the previous two equalities.
\end{proof}
\begin{remark}
A quick application of Plancherel's identity implies the following integration by part formula that will be used later: if $U$ is a pencil-beam with null interior source and boundary condition $G=\delta(X'-Y')\delta(V-W)$, then for any $\Phi\in C(\RR^{n}_+\times\RR^{n-1}))\cap H^1(\RR_+;L^2(\RR^{n-1}\times\RR^{n-1}))$, 
\begin{equation}\label{IBPformula}
\begin{aligned}
\int_{\RR^{n}_+\times\RR^{n-1}}U\partial_{X^n}\Phi dX'dX^ndV &= -\int_{\RR^{n}_+\times\RR^{n-1}}\partial_{X^n}U\Phi dX'dX^ndV- \Phi(Y',0,W).\\
\end{aligned}
\end{equation}
\end{remark}
\medskip

The approximation analysis of the next section requires estimates on various derivatives of the solutions to Fermi pencil-beam equation. This is not an issue for large values of $(X,V)$ due to the exponential decay at infinity but becomes more tedious when approaching $X^n=0$, where the pencil-beam solutions are singular.  The precise integrals that need to be controlled are summarized in the following lemma.

\begin{lemma}\label{lemma:U_integrals}
Let $U$ be a solution to \eqref{FPb}-\eqref{FPb_bc} for $F=0$ and $G=\delta(X')\delta(V-\Theta)$, for some $\Theta\in\RR^{n-1}$, and coefficients $\tilde{\sigma}\in C^3(\bar{\RR}_+)$, $\tilde{\lambda}\in C(\bar{\RR}_+)$. The integrals 
\begin{equation*}\label{norms_to_estimate_lemma}
\||X^i-X^n\Theta^i|^l|V^j|^m\partial_{X'}^p\partial_{V}^q U\|_{L^1},\quad i,j\in\{1,\dots,n-1\}
\end{equation*}
are finite:
\begin{itemize}
\item[(i)] in general for any integers $l,m\geq 0$ and multi-indices $p,q\geq 0$ so that $3l\geq 3|p|+|q|$;
\item[(ii)] and for any $l,m\geq 0$ and $p,q\geq 0$ so that $3l+m\geq 3|p|+|q|$ if $\Theta=0$.
\end{itemize}
Furthermore, the same conclusion holds for a solution $\bfU$ to \eqref{FPb}-\eqref{FPb_bc} with $G=0$ and $F=((X'-X^n\Theta)\cdot\xi(X^n))\partial^q_VU$, $|q|\leq 3$ and $\xi$ continuous and bounded in $\bar{\RR}_+$ (this will be used in the proof of lemma \ref{lemma:ord1_approx}).
\end{lemma}
\begin{proof}
Let us write $\eta = \eta(X^n):= e^{-\int_0^{X^n}\tilde{\lambda}(s)ds}$  so that for $F=0$ and $G=\delta(X')\delta(V-\Theta)$ the Fermi pencil-beam equation can be written as
\begin{equation}\label{pencil_beam}
\begin{aligned}
U(X,V)=\eta(X^n)\tau_{X^nV}{\bf H}_1(X,V-\Theta).
\end{aligned}
\end{equation}
For any polynomial $p(X',V)$ and function $f(X',V)$ such that the following quantities are finite, we have
\begin{equation*}
\int p(X',V)[\tau_{X^nV}f](X',V)dX'dV = \int p(X'+X^nV,V)f(X',V)dX'dV,
\end{equation*}
and thus, the previous identity along with a change of variables gives
\begin{equation}\label{integrals_H1}
\begin{aligned}
&\||X^i-X^n\Theta^i|^l|V^j|^m\partial_{X'}^p\partial_{V}^q U\|_{L^1}\\
&\leq \sum_{\tilde{q}\leq q} \binom{q}{\tilde{q}}\|\eta(X^n)|X^i+X^nV^i)|^l|V^j+\Theta^j|^m(X^n)^{|\tilde{q}|}\partial_{X'}^{p+\tilde{q}}\partial_{V}^{q-\tilde{q}} \bH_1(X,V)\|_{L^1}.\\
\end{aligned}
\end{equation}
We write the kernel $\bH_1(X,V)$ as
\begin{equation}\label{kernel_H1}
\begin{aligned}
\bH_1(X,V) &=\frac{1}{(4\pi\sqrt{\Delta})^{n-1}}e^{-(\alpha|X'|^2 + 2\beta X'\cdot V + \gamma|V|^2)},
\end{aligned}
\end{equation}
where $\Delta = ac-b^2>0$, 
\begin{equation}\label{coeff_kernel}
a(X^n) = \int^{X^n}_0\tilde{\sigma}(s)ds,\quad b(X^n) = \int^{X^n}_0s\tilde{\sigma}(s)ds,\quad c(X^n) = \int^{X^n}_0s^2\tilde{\sigma}(s)ds,
\end{equation}
with $\tilde{\sigma}>0$, and  
\begin{equation*}
\alpha = \frac{a}{4\Delta},\quad \beta = \frac{b}{4\Delta}\quad\text{and}\quad \gamma = \frac{c}{4\Delta}.
\end{equation*}
Near $X^n=0$ and for any extension of $\tilde{\sigma}\in C^3([0,\infty))$ to $\RR$, we deduce from Taylor's theorem that there is $h_0>0$ and a function $h:\RR\to\RR$ such that
\begin{equation}\label{Taylor_Delta}
ac-b^2 = (X^n)^4(h_0 + h(X^n)),
\end{equation}
with $h(X^n)\to 0$ as $X^n\to 0$. In addition, in the limits $X^n\to 0$ and $X^n\to\infty$, the coefficients $a$, $b$ and $c$ behave respectively as $X^n$, $(X^n)^2$ and $(X^n)^3$, and this implies 
\begin{equation*}
\alpha\approx (X^n)^{-3},\quad\beta\approx (X^n)^{-2},\quad\text{and}\quad \gamma \approx (X^n)^{-1}.
\end{equation*}

To estimate the terms on the right hand side of \eqref{integrals_H1} we need upper bounds for integrals of the form $I_k:=\int^{\infty}_{-\infty}|t|^ke^{-\alpha t^2+2\beta t}dt$, for $k\geq 0$. We list a few of them:
\begin{equation*}
I_k = \left\{
\begin{array}{ll}
\sqrt{\frac{\pi}{\alpha}}e^{\frac{\beta^2}{\alpha}}&k=0,\\
\frac{1}{\alpha}(1+\beta\sqrt{\frac{\pi}{a}}\text{\rm erf}\big(\frac{\beta}{\sqrt{\alpha}}\big)e^{\frac{\beta^2}{\alpha}})\;\leq\; \frac{1}{\alpha}(1+|\beta|\sqrt{\frac{\pi}{\alpha}}e^{\frac{\beta^2}{\alpha}})&k=1,\\
\frac{\sqrt{\pi}(2\alpha+4\beta^2)}{4\alpha^{5/2}}e^{\frac{\beta^2}{\alpha}}&k=2,\\
\frac{1}{\alpha^2}\big(1+\frac{\beta^2}{\alpha} +(\frac{2\beta^2}{\alpha}+3)\frac{\beta}{2}\sqrt{\frac{\pi}{\alpha}}\text{\rm erf}\big(\frac{\beta}{\sqrt{\alpha}}\big)e^{\frac{\beta^2}{\alpha}} \big)&k=3.\\
\hspace{5em}\leq \frac{1}{\alpha^2}\big(1+\frac{\beta^2}{\alpha} +(\frac{2\beta^2}{\alpha}+3)\frac{|\beta|}{2}\sqrt{\frac{\pi}{\alpha}}e^{\frac{\beta^2}{\alpha}} \big)
\end{array}
\right.
\end{equation*}
Performing the integrations over $X'$ and $V$ and using the values of the integrals above one can see that each factor $|X^i|$ scales in the estimates as $(X^n)^{3/2}$ since it bring a division by a factor $\sqrt{\alpha}$, while each $|V^j|$ scales as $(X^n)^{1/2}$ (division by $\sqrt{\gamma}$). Symmetrically, every derivative taken with respect to an $X^i$ bring a factor $\alpha X^i +\beta V^j$ to the estimates which after integration translates into a division by $(X^n)^{3/2}$, and similarly every derivative with respect to $V^j$ brings to the estimates a factor $\beta X^j + \gamma V^j$ which leads to a division by $(X^n)^{1/2}$. The condition imposed on exponent and multi-indices in \eqref{norms_to_estimate_lemma} then implies that after integrating with respect to $X'$ and $V$, integrals of the form
$$
\|(X^n)^{|\tilde{q}|}|X^i|^{l}|V^j|^{\tilde{m}}\partial_{X'}^{p+\tilde{q}}\partial_{V}^{q-\tilde{q}} U\|_{L^1}\quad\text{and}\quad\|(X^n)^{l+|\tilde{q}|}|V^j|^{l+\tilde{m}}\partial_{X'}^{p+\tilde{q}}\partial_{V}^{q-\tilde{q}} U\|_{L^1},
$$
with $\tilde{m} \in\{0,m\}$, are bounded by a constant factor times $\int^\infty_0\eta(X^n)(X^n)^{s}dX^n$, 
for some nonnegative rational number $s$ depending on $l,m,p$ and $q$. 
The right hand side is finite since $s$ is nonnegative and the exponential decay at infinity of $\eta$.

The two cases in the statement of the lemma derive from the previous estimation by noticing that for $\Theta = 0$ there is no integral with $\tilde{m}=0$, and we just need $3l+m\geq 3|p|+|q|$ instead of the stronger requirement $3l\geq 3|p|+|q|$.

The second part of the proof consist in estimating the same integrals with $U$ replaced by the more intricate function $\bfU$. We write
\begin{equation*}
\begin{aligned}
\bfU(X,V)= \int^{X^n}_0\eta_t(X^n)\tau_{X^nV}({\bf H}_2(X^n;t)*\tau_{-tV}F(t)dt,
\end{aligned}
\end{equation*}
for $\eta_t := e^{-\int^{X^n}_t\tilde{\lambda}(s)ds}$ and kernel
\begin{equation}\label{kernel_H2}
\bH_2(X,V;t) = \frac{1}{(4\pi\sqrt{\Delta_t})^{n-1}}e^{-(\alpha_t|X'|^2 + 2\beta_tX'\cdot V + \gamma_t|V|^2)},\quad t\in(0,X^n).
\end{equation}
Here we use the notation $\Delta_t = a_tc_t-b_t^2$ for
\begin{equation}\label{coeff_kernel_t}
a_t(X^n) = \int^{X^n}_t\tilde{\sigma}(s)ds,\quad b_t(X^n) = \int^{X^n}_ts\tilde{\sigma}(s)ds,\quad c_t(X^n) = \int^{X^n}_ts^2\tilde{\sigma}(s)ds,
\end{equation} 
and $\alpha_t = \frac{a_t}{4\Delta_t}$, $\beta_t = \frac{b_t}{4\Delta_t}$, $\gamma_t = \frac{c_t}{4\Delta_t}$. 

Let us analyze $F(t)=(X'-t\Theta)\cdot\xi(t)\partial^r_{V}U(t)$ (with $|r|\leq 3$) first. We see that 
\begin{equation*}
\begin{aligned}
\tau_{-tV}F(t) &= \tau_{-tV}[(X'-t\Theta)\cdot\xi(t)\partial^r_V(\eta(t)\tau_{tV}\bH_1(X',t,V-\Theta))]\\
&= \tau_{-tV}[(X'-t\Theta)\cdot\xi(t)\eta(t) \sum_{\tilde{r}\leq r}\binom{r}{\tilde{r}}t^{|\tilde{r}|}\tau_{tV}[\partial_{X'}^{\tilde{r}}\partial_V^{r-\tilde{r}}\bH_1(X',t,V-\Theta)]]\\
&= \sum_{k=1}^{n-1}\sum_{\tilde{r}\leq r}\binom{r}{\tilde{r}}t^{|\tilde{r}|}\eta(t)\xi^k(t)(X^k+t(V^k-\Theta^k))\partial_{X'}^{\tilde{r}}\partial_V^{r-\tilde{r}}\bH_1(X',t,V-\Theta).
\end{aligned}
\end{equation*}
Denoting by $e_k$ the multi-index with a 1 in the $k$th-position and the rest all zeros, and also writing $r = (r_1,\dots,r_{n-1})$, then
\begin{equation*}
\begin{aligned}
&(X^k+t(V^k-\Theta^k))\partial_{X'}^{\tilde{r}}\partial_V^{r-\tilde{r}}\bH_1(X',t,V-\Theta) \\
&\hspace{5em}= \partial_{X'}^{\tilde{r}}\partial_V^{r-\tilde{r}}[(X^k+t(V^k-\Theta^k))\bH_1(X',t,V-\Theta)] \\
&\hspace{6em}-(1-\delta_{0\tilde{r}_k})\partial_{X'}^{\tilde{r}-e_k}\partial_V^{r-\tilde{r}}\bH_1(X',t,V-\Theta)\\
&\hspace{6em}-(1-\delta_{0(r_k-\tilde{r}_k)})t\partial_{X'}^{\tilde{r}}\partial_V^{r-\tilde{r}-e_k}\bH_1(X',t,V-\Theta),
\end{aligned}
\end{equation*}
with $\delta_{ij}$ the Kronecker delta. Convolving $\tau_{-tV}F(t)$ with $\bH_2$ gives several terms of the form
\begin{equation*}
\begin{aligned}
&{\bf H}_2*t^{|\tilde{r}|+|s_2|}\eta(t)\xi^i(t)\partial_{X'}^{\tilde{r}-s_1}\partial_V^{r-\tilde{r}-s_2}[(X^k+t(V^k-\Theta^k))^{1-|(s_1,s_2)|}\bH_1(X',t,V-\Theta)]
\end{aligned}
\end{equation*}
with $|(s_1,s_2)|\leq1$. 
Therefore, the estimation of $\||X^i-X^n\Theta^i|^l|V^j|^m\partial_{X'}^p\partial_{V}^q \bfU\|_{L^1}$ reduces to obtaining upper bounds for
\begin{equation}\label{est_1}
\begin{aligned}
&\big\||X^i+X^n(V^i-\Theta^i)|^l|V^j|^m(X^n)^{|\tilde{q}|}\int^{X^n}_0t^{|\tilde{r}|}\eta(t)\xi^k(t)\eta_t(X^n) \\
&\hspace{0em}\times \partial_{X'}^{p+\tilde{q}}\partial_{V}^{q-\tilde{q}}\big({\bf H}_2*\partial_{X'}^{\tilde{r}-s_1}\partial_V^{r-\tilde{r}-s_2}[(X^k+t(V^k-\Theta^k))^{1-|(s_1,s_2)|}\bH_1(X',t,V-\Theta)]\big) dt\big\|_{L^1}
\end{aligned}
\end{equation}
with $i,j,k\in\{1,\dots,n-1\}$, and multi-indices $\tilde{p}\leq p$, $\tilde{q}\leq q$, $\tilde{r}\leq r$, $|(s_1,s_2)|\leq 1$.

Notice that
\begin{equation*}
\begin{aligned}
&\partial_{X'}^{p+\tilde{q}}\partial_{V}^{q-\tilde{q}} \big({\bf H}_2*\partial_{X'}^{\tilde{r}-s_1}\partial_V^{r-\tilde{r}-s_2}[(X^k+t(V^k-\Theta^k))^{1-|(s_1,s_2)|}\bH_1(X',t,V-\Theta)] \big)\\
&\hspace{1em}= \partial_{X'}^{p+\tilde{q}+\tilde{r}-s_1}\partial_{V}^{q-\tilde{q}+r-\tilde{r}-s_2} \big({\bf H}_2*[(X^k+t(V^k-\Theta^k))^{1-|(s_1,s_2)|}\bH_1(X',t,V-\Theta)] \big).
\end{aligned}
\end{equation*}
On the other hand, denoting by $\Sigma=\left(\begin{matrix}\alpha&\beta\\\beta&\gamma\end{matrix}\right)$, then $\nabla_{X^k,V^k}\bH_1 = -2\bH_1\Sigma\left(\begin{matrix}X^k\\V^k\end{matrix}\right)$ and consequently
\begin{equation*}
(X^k+tV^k)\bH_1=-\frac{1}{2}\Big((\Sigma^{-1}\nabla_{X^k,V^k}\bH_1)_1+t(\Sigma^{-1}\nabla_{X^k,V^k}\bH_1)_2\Big),
\end{equation*}
where $\Sigma^{-1} = 4\left(\begin{matrix}c&-b\\-b&a\end{matrix}\right)$. This yields
\begin{equation*}
\begin{aligned}
(X^k+t(V^k-\Theta^k))\bH_1(X',t,V-\Theta) &= 2(tb(t)-c(t))\partial_{X^k}\bH_1(X',t,V-\Theta)\\
&\hspace{1em}-2(ta(t)-b(t))\partial_{V^k}\bH_1(X',t,V-\Theta).
\end{aligned}
\end{equation*}
Therefore,
\begin{equation*}
\begin{aligned}
&\partial_{X'}^{p+\tilde{q}}\partial_{V}^{q-\tilde{q}} \big({\bf H}_2(X^n;t)*\partial_{X'}^{\tilde{r}-s_1}\partial_V^{r-\tilde{r}-s_2}[(X^k+t(V^k-\Theta^k))^{1-|(s_1,s_2)|}\bH_1(X',t,V-\Theta)] \big)\\
&\hspace{0em}= \left\{\begin{array}{ll}
\partial_{X'}^{p+\tilde{q}+\tilde{r}-s_1}\partial_{V}^{q-\tilde{q}+r-\tilde{r}-s_2} \big({\bf H}_2*\bH_1(X',t,V-\Theta) \big),&\hspace{-.5em}|(s_1,s_2)|=1,\\
&\\
2(tb(t)-c(t))\partial_{X'}^{p+\tilde{q}+\tilde{r}+e_k}\partial_{V}^{q-\tilde{q}+r-\tilde{r}} \big({\bf H}_2*\bH_1(X',t,V-\Theta) \big)\\
\hspace{1em} -2(ta(t)-b(t))\partial_{X'}^{p+\tilde{q}+\tilde{r}}\partial_{V}^{q-\tilde{q}+r-\tilde{r}+e_k} \big({\bf H}_2*\bH_1(X',t,V-\Theta) \big),& \hspace{-.5em}|(s_1,s_2)|=0.
\end{array}
\right.
\end{aligned}
\end{equation*}
The factors $(tb(t)-c(t))$ and $(ta(t)-b(t))$ behave respectively as $t^3$ and $t^2$, in the limit $t\to 0$. 

We now write $\bH = \bH_2*\bH_1$, therefore
\begin{equation}\label{kernel_H4}
\bH(X,V;t) = \frac{1}{(4\pi\sqrt{\hat{\Delta}})^{n-1}}e^{-\frac{1}{4\Delta}(\hat{a}|X'|^2 + 2\hat{b}X'\cdot V + \hat{c}|V|^2)},\quad t\in(0,X^n),
\end{equation}
for $\hat{\Delta} = \hat{a}\hat{c}-\hat{b}^2$ and
\begin{equation*}
\hat{a}(X^n;t) =a(X^n)+a_t(X^n),\;\; \hat{b}(X^n;t)=b(X^n)+b_t(X^n) ,\;\; \hat{c}(X^n;t)=c(X^n)+c_t(X^n) .
\end{equation*}
Noting that $\eta(X^n) = \eta_t(X^n)\eta(t)$ and recalling that $\xi(t)$ is bounded, we deduce that the estimation of \eqref{est_1} reduces to bounding from above several integrals of the form
\begin{equation*}
\begin{aligned}
&\int_0^\infty \eta(X^n)(X^n)^{|\tilde{q}|} \\
&\times\int_0^{X^n} t^{|\tilde{r}|}\||X^i+X^n(V^i-\Theta^i)|^l|V^j|^mP_{X',V}\bH(X',X^n,V-\Theta;t)\|_{L^1_{X',V}}dtdX^n,
\end{aligned}
\end{equation*}
with $P_{X',V}$ any of the following differential operators:
\begin{equation*}
\begin{aligned}
&\partial_{X'}^{p+\tilde{q}+\tilde{r}-s_1}\partial_{V}^{q-\tilde{q}+r-\tilde{r}-s_2},\quad |(s_1,s_2)|=1\\
&2(tb(t)-c(t))\partial_{X'}^{p+\tilde{q}+\tilde{r}+e_k}\partial_{V}^{q-\tilde{q}+r-\tilde{r}}\quad\text{and}\quad 2(ta(t)-b(t))\partial_{X'}^{p+\tilde{q}+\tilde{r}}\partial_{V}^{q-\tilde{q}+r-\tilde{r}+e_k}.
\end{aligned}
\end{equation*}
These last expressions are bounded by the same arguments as in the first part of the proof.
\end{proof}

We conclude this section with a few words regarding the adjoint problem associated to \eqref{FPb}-\eqref{FPb_bc}.
Following an analogous argument as in the existence result for equation \eqref{FPb}, if we are given a bounded and compactly supported function $\Psi$, the solution to the {\em backward Fermi pencil-beam} system
\begin{equation}\label{backward_FPB}
- \tilde{\sigma}\Delta_V W  - V \cdot\nabla_{X'} W- \partial_{X^n}W  + \tilde{\lambda}W= \Psi,\quad (X,V)\in \RR^{n}_+\times\RR^{n-1},
\end{equation}
augmented with a vanishing condition at infinity
\begin{equation}\label{backward_FPB_bc}
\lim_{X^n\to\infty}W = 0,
\end{equation}
can be written in the form
\begin{equation*}
W(X,V) = \int_{X^n}^\infty e^{-\int^t_{X^n}\tilde{\lambda}(s)ds}\tau_{-X^nV}({\bf H}_3(t)*\tau_{tV}\Psi(t))dt,
\end{equation*}
for a given Gaussian kernel ${\bf H}_3$ defined similarly as $\bH_1$ in \eqref{def:H1} but with coefficients $E_k$ replaced by $\int^t_{X^n}s^k\tilde{\sigma}(s)ds$. Formula (4.18) is derived similarly as in the proof of Proposition 4.1 and the
Xn
details of this can be found in Appendix \ref{appdx:bFPBsolution}.

Using this explicit expression we directly obtain Lipschitz continuity of $W$. For a proof of this we refer the reader to Appendix \ref{appdx:proof_lemma_FPb_Lip}.
\begin{lemma}\label{lemma:FPb_Lip}
Let $\Psi\in C_c(\bar{Q})$ and Lipschitz continuous with respect to the variables $Z=(X',V)$. There exists $C>0$ so that, for $W$ solution to \eqref{backward_FPB}-\eqref{backward_FPB_bc},
\begin{equation*}
|W(Z_1,X^n) -W(Z_2,X^n) |\leq C\Lip(\Psi)|Z_1-Z_2|,
\end{equation*}
for all $Z_1,Z_2\in\RR^{2(n-1)}$ and $X^n>0$.
\end{lemma}
\section{1-Wasserstein comparison analysis}\label{sec:1W_analysis}
\subsection{Approximation via pencil-beams}\label{sec:Pb_approx}
We now consider the approximation of narrow beam solutions to the Fokker-Planck equation by ballistic transport and Fermi pencil-beams, which correspond to the limit $\epsilon\ll1$. The modeling of narrow beams is best modeled as beams propagating in a half space with singular boundary sources. Heuristically, each delta source term gives rise to Fokker-Planck, Fermi pencil-beam, and ballistic transport solutions. More general source terms may then be modeled as superpositions of such delta sources. We first state our comparison results for delta source terms and then consider some models of linear superposition.  As we mentioned in the introduction, all comparisons are obtained in the well-adapted notion of 1-Wasserstein distance.

\subsubsection{Delta boundary source}
We consider first the case of a delta incoming boundary condition
\begin{equation}\label{def:delta_source}
g(x,\theta) = \delta(x)\delta_{\SS^{n-1}}(\theta-\eta),\quad\text{for}\quad\eta\in\SS^{n-1}\quad\text{such that}\quad|N-\eta|=O(\epsilon^2),
\end{equation}
where $N=-\nu(0)=(0,\dots,0,1)$ and we let $u$ be the solution to \eqref{FP}-\eqref{FP_bc} from Theorem \ref{thm:FP_dist}. 

Associated to $\eta=(\eta',\eta^n)$, whose stereographic projection with respect to the south pole (see \eqref{def:sproj}) is given by
\begin{equation*}
\mathcal{S}(\eta) = \frac{\eta'}{1+\eta^n}\in \RR^{n-1},
\end{equation*}
we  set
\begin{equation*}
G(X',V)=\delta(X')\delta(V-\epsilon^{-1}\mathcal{S}(\eta)),\quad \tilde{\sigma}=\frac{1}{4}\sigma(X^n\eta)\quad\text{and}\quad \tilde{\lambda}=\lambda(X^n\eta),
\end{equation*}
where the coefficients $\tilde{\sigma}$ and $\tilde{\lambda}$ have an indirect dependence on $\epsilon$ through the direction $\eta$. Letting $U$ be the solution to \eqref{FPb} with the above choice of parameters, we define the {\em pencil-beam approximation} as
\begin{equation*}
\mathfrak{u}(x,\theta) := (2\epsilon)^{-2(n-1)}U((2\epsilon)^{-1}x',x^n,\epsilon^{-1}\mathcal{S}(\theta)).
\end{equation*}
Throughout this section, we assume $\sigma\in C^3(\bar{\RR}^n_+)$ and $\lambda \in C^2(\bar{\RR}^n_+)$. The extra regularity of $\sigma$ is needed to control the integrals that appear when applying Lemma \ref{lemma:U_integrals}.

We then have the following approximation result:
\begin{theorem}\label{thm:main1}There exists a positive constant $C=C(\|\sigma\|_{C^3},\|\lambda\|_{C^2})$ such that, for $\kappa\gtrsim 1$,
\begin{equation*}
\mathcal{W}^1_{\kappa}(u,\mathfrak{u})\leq C\epsilon^2\kappa.
\end{equation*}
\end{theorem}

\begin{proof}
Consider an open and bounded set $\Omega\subset\bar{\RR}^n_+$ and let $\psi\in BL_{1.\kappa}(\Omega\times\SS^{n-1})$ which we extend to the entire half-space as a continuous and compactly supported function preserving its Lipschitz constant. Let $\varphi$ the unique (strong) solution to the backward Fokker-Planck system,
\begin{equation}\label{backward_FP}
-\epsilon^2\sigma\Delta_\theta\varphi- \theta\cdot\nabla_x\varphi + \lambda\varphi = \psi,\quad \varphi|_{\Gamma_+} = 0,
\end{equation}
which is continuous in $Q\cup\Gamma_-$ due to Theorem  \ref{thm:Holder_reg} and the remark after it.
The goal is to show that 
\begin{equation*}
\int_{\Omega\times\SS^{n-1}}\psi(x,\theta) (u-\mathfrak{u})=O(\epsilon^2\kappa)   +o(1),
\end{equation*}
where the error term goes to zero as $\Omega\to\bar{\RR}_+$ since it is bounded by $\int_{\Omega^c\times\SS^{n-1}}( u +\mathfrak{u})dxd\theta$. It appears due to the fact that we are extending $\psi$ in space.
From the definition of weak solutions, the left hand side of the above expression reads 
\begin{equation*}
\begin{aligned}
\int_{Q}\psi(x,\theta) d(u-\mathfrak{u})&=  \langle g,\varphi|\theta\cdot\nu(x)|\rangle_{\Gamma_-} - \int_{Q}\mathfrak{u}(x,\theta)\psi(x,\theta)dxd\theta \\
&=\varphi(0,\eta)\eta^n - \int_{Q}\mathfrak{u}(x,\theta)\psi(x,\theta) dxd\theta,
\end{aligned}
\end{equation*}
and consequently our objective is to show that
\begin{equation}\label{plan_proof1}
\int_{Q}\mathfrak{u}
\psi dxd\theta= \varphi(0,\eta) \eta^n+ O(\epsilon^{2}\kappa).
\end{equation}
In order to achieve this we rescale the equation and pass to the stretched coordinates (at the level of the diffusion), which we recall is defined by the equality
\begin{equation*}
(x',x^n,\theta) = (2\epsilon X',X^n,\mathcal{J}(\epsilon V)),
\end{equation*}
with $\mathcal{J}=\mathcal{S}^{-1}$ the inverse of the stereographic projection from the south pole (defined in \eqref{def:sproj}-\eqref{def:isproj}).

We then set $\Phi$ and $\Psi$ as the following rescalings of $\varphi$ and $\psi$, respectively, which preserve the $L^\infty$-norm:
\begin{equation}\label{def:Phi_Psi}
\varphi(x,\theta) = \Phi((2\epsilon)^{-1}x',x^n,\epsilon^{-1}\mathcal{S}(\theta))\quad\text{and}\quad \psi(x,\theta) = \Psi((2\epsilon)^{-1}x',x^n,\epsilon^{-1}\mathcal{S}(\theta)).
\end{equation}
By noticing
\begin{equation*}
\begin{aligned}
\int_{Q}\mathfrak{u}(x,\theta)\psi(x,\theta)dxd\theta &=  \int_{Q}\mathfrak{u}(-\epsilon^2\sigma\Delta_\theta \varphi -\theta \cdot\nabla_x\varphi+ \lambda \varphi )dxd\theta\\
&=-\epsilon^2\int_{Q}\sigma (\Delta_\theta\mathfrak{u})\varphi dxd\theta - \int_{Q}\mathfrak{u}(\theta \cdot\nabla_x\varphi )dxd\theta + \int_{Q}\lambda\mathfrak{u}\varphi dxd\theta,
\end{aligned}
\end{equation*}
we pass to stretched coordinates and analyze individually every term on the right hand side. The first one gives
\begin{equation*}
\begin{aligned}
&-\epsilon^2\int_{Q}\sigma (\Delta_\theta\mathfrak{u})\varphi dxd\theta \\
&\hspace{2em}= -\int_{\mathcal{Q}} \tilde{\sigma}(X^n)(\Delta_V U)\Phi dXdV\\
&\hspace{3em}+  \int _{\mathcal{Q}}\frac{1}{4}\sigma(2\epsilon X',X^n)\Big(1 - 2^{3-n}c^{n-3}(\epsilon V)\Big)(\Delta_VU) \Phi dXdV\\
& \hspace{3em}-\epsilon(n-3)\int_{\mathcal{Q}} c^{-3}(\epsilon V)\sigma(2\epsilon X',X^n)\big((\nabla c)(\epsilon V)\cdot\nabla_V U\big)\Phi2^{1-n}c^{n-1}(\epsilon V)dXdV\\
&\hspace{3em}+\frac{1}{4}\int_{\mathcal{Q}} \big(\sigma( X^n\eta) - \sigma(2\epsilon X',X^n)\big)(\Delta_V U)\Phi dXdV.
\end{aligned}
\end{equation*}
For the second term we have
\begin{equation*}
\int_{Q}\mathfrak{u}(\theta \cdot\nabla_x\varphi )dxd\theta = 
 \int_{\mathcal{Q}} \Big[\Big(\frac{V\cdot \nabla_{X'} \Phi}{1+\epsilon^2|V|^2}\Big)U + \Big(\frac{1-\epsilon^2|V|^2}{1+\epsilon^2|V|^2}\Big)(\partial_{X^n}\Phi)U \Big]\frac{c^{n-1}(\epsilon V)}{2^{n-1}}dXdV.
\end{equation*}
Thus, after integration by parts (where we use \eqref{IBPformula} for the $X^n$-variable) we obtain
\begin{equation*}
\begin{aligned}
-\int_{Q}\mathfrak{u}(\theta \cdot\nabla_x\varphi )dxd\theta&\\
&\hspace{-5em} =  \int_{\mathcal{Q}} \Big[\Big(\frac{V\cdot \nabla_{X'} U}{1+\epsilon^2|V|^2}\Big)\Phi + \Big(\frac{1-\epsilon^2|V|^2}{1+\epsilon^2|V|^2}\Big)(\partial_{X^n}U)\Phi \Big]\frac{c^{n-1}(\epsilon V)}{2^{n-1}}dXdV\\
&\hspace{-4em}+\Phi(0,0,\epsilon^{-1}\mathcal{S}(\eta))\underbrace{\Big(\frac{1-|\mathcal{S}(\eta)|^2}{1+|\mathcal{S}(\eta)|^2}\Big)\langle\mathcal{S}(\eta)\rangle^{-2(n-1)}}_{=\eta^n(1 + O(\epsilon^{2}))\quad (\text{due to \eqref{def:delta_source}})}\\
&\hspace{-5em}= \int_{\mathcal{Q}}(V\cdot\nabla_{X'}U)\Phi dXdV +\Phi(0,0,\epsilon^{-1}\mathcal{S}(\eta)) \eta^n + O(\epsilon^2)\\
&\hspace{-4em}+ \int_{\mathcal{Q}} \Big(\frac{2^{1-n}c^{n-1}(\epsilon V)}{1+\epsilon^2|V|^2}-1\Big)(V\cdot \nabla_{X'}U) \Phi dXdV\\
& \hspace{-4em}+ \int_{\mathcal{Q}} \Big(\frac{1-\epsilon^2|V|^2}{1+\epsilon^2|V|^2}2^{1-n}c^{n-1}(\epsilon V)-1\Big)(\partial_{X^n}U)\Phi dXdV.
\end{aligned}
\end{equation*}
Finally, the third term gives
\begin{equation*}
\begin{aligned}
\int_{Q}\lambda\mathfrak{u}\varphi dxd\theta =& \int_{\mathcal{Q}} \tilde{\lambda} U\Phi dXdV + \int_{\mathcal{Q}}\lambda(2\epsilon X',X^n)(2^{1-n}c^{n-1}(\epsilon V) - 1) U\Phi dXdV\\
&+\int_{\mathcal{Q}} (\lambda( 2\epsilon X',X^n) - \lambda( X^n\eta)U\Phi dXdV.
\end{aligned}
\end{equation*}
Summarizing all the above and recalling that $U$ is a solution to the Fermi-equation (i.e.  $\mathcal{P}(U)=0$) we have
\begin{equation*}
\begin{aligned}
\int_{Q}\mathfrak{u}(x,\theta)\psi(x,\theta)dxd\theta
&= \varphi(0,\eta)\eta^n + J_1(\Phi) +J_2(\Phi) + O(\epsilon^2),
\end{aligned}
\end{equation*}
with
\begin{equation}\label{J1}
\begin{aligned}
J_1(\Phi) &=   \int _{\mathcal{Q}}\frac{1}{4}\sigma(2\epsilon X',X^n)\Big(1 - 2^{3-n}c^{n-3}(\epsilon V)\Big)(\Delta_VU) \Phi dXdV\\
&\quad  -\epsilon(n-3)\int_{\mathcal{Q}} c^{-3}(\epsilon V)\sigma(2\epsilon X',X^n)\big((\nabla c)(\epsilon V)\cdot\nabla_V U\big)\Phi2^{1-n}c^{n-1}(\epsilon V)dXdV \\
&\quad + \int_{\mathcal{Q}} \Big(\frac{2^{1-n}c^{n-1}(\epsilon V)}{1+\epsilon^2|V|^2}-1\Big)(V\cdot \nabla_{X'}U) \Phi dXdV\\
&\quad + \int_{\mathcal{Q}} \Big(\frac{1-\epsilon^2|V|^2}{1+\epsilon^2|V|^2}2^{1-n}c^{n-1}(\epsilon V)-1\Big)(\partial_{X^n}U)\Phi dXdV\\
&\quad + \int_{\mathcal{Q}}\lambda(2\epsilon X',X^n)(2^{1-n}c^{n-1}(\epsilon V) - 1) U\Phi dXdV\\
&= \sum_{i=1}^5J_{1,i},
\end{aligned}
\end{equation}
and 
\begin{equation}\label{J2}
\begin{aligned}
J_2(\Phi) &=  \frac{1}{4}\int_{\mathcal{Q}} (\sigma(X^n\eta) - \sigma( 2\epsilon X',X^n))(\Delta_V U)\Phi dXdV\\
&\quad -\int_{\mathcal{Q}} (\lambda(X^n\eta) - \lambda( 2\epsilon X',X^n))U\Phi dXdV.
\end{aligned}
\end{equation}
To estimate $|J_1|$ we notice that for $m\geq n-3$ ($n\geq 2$), 
\begin{equation}\label{aux1}
1-2^{-m}c^m(\epsilon V) = \left\{\begin{array}{ll}
\epsilon^2\sum^m_{k=1}\frac{|V|^{2}}{(1+\epsilon^2|V|^2)^k},\quad & m>0\footnotemark\\
0,& m=0\\
-\epsilon^2|V|^2,& m=-1,
\end{array}\right.
\end{equation}
\footnotetext{The equality follows from the difference of powers formula: $x^{p+1} - y^{p+1}=(x-y)(x^p+x^{p-1}y + \dots +xy^{p-1}+y^p)$.}
and thus for all $m\geq n-3$,
\begin{equation*}
|1-2^{-m}c^m(\epsilon V) |\leq |m|\epsilon^2|V|^2.
\end{equation*}
Furthermore, $(\nabla c)(\epsilon V) = -\frac{4\epsilon V}{(1+\epsilon^2|V|^2)^2}$ and
\begin{equation*}
\frac{c^{n-4}(\epsilon V)}{2^{n-4}(1+\epsilon^2|V|^2)^2} = \frac{1}{(1+\epsilon^2|V|^2)^{n-2}}\leq 1.
\end{equation*}
We then have that
\begin{equation*}
\begin{aligned}
&|J_{1,1}|+|J_{1,2}|+|J_{1,5}| \\
&\hspace{3em}\leq C\epsilon^2\|\Phi\|_\infty\big(|n-3|\||V|^2\Delta_VU\|_{L^1(\RR^{n}_+\times\RR^{n-1})} \\
&\hspace{9em}+ \|V\cdot\nabla_VU\|_{L^1(\RR^{n}_+\times\RR^{n-1})}+\|U\|_{L^1(\RR^{n}_+\times\RR^{n-1})}\big).
\end{aligned}
\end{equation*}
The term $J_{1,3}$ is bounded since
\begin{equation*}
\Big|\frac{2^{1-n}c^{n-1}(\epsilon V)}{1+\epsilon^2|V|^2}-1\Big| = \epsilon^{2}|V|^2\sum_{k=1}^{n}\frac{1}{1+\epsilon^2|V|^2}\leq n\epsilon^2|V|^2,
\end{equation*}
while the estimate for $J_{1,4}$ follows from
\begin{equation*}
\Big|\frac{1-\epsilon^2|V|^2}{1+\epsilon^2|V|^2}2^{1-n}c^{n-1}(\epsilon V)-1\Big|\leq \Big(1-\frac{1}{(1+\epsilon^2|V|^2)^n}\Big) + \frac{\epsilon^2|V|^2}{(1+\epsilon^2|V|^2)^n}\leq (n+1)\epsilon^2|V|^2.
\end{equation*}
We obtain
\begin{equation*}
\begin{aligned}
|J_{1,3}|+|J_{1,4}| &\leq C\epsilon^2\|\Phi\|_\infty\big(n\||V|^2V\cdot\nabla_{X'}U\|_{L^1(\RR^{n}_+\times\RR^{n-1})} \\
&\hspace{6em}+ (n+1)\||V|^2\partial_{X^n}U\|_{L^1(\RR^{n}_+\times\RR^{n-1})}\big),
\end{aligned}
\end{equation*}
where we use the Fermi Pencil-beam equation to get
\begin{equation*}
\||V|^2\partial_{X^n}U\|_{L^1} \leq \||V|^2\Delta_VU\|_{L^1} + \| |V|^2V\cdot \nabla_{X'}U\|_{L^1} +\||V|^2U\|_{L^1}.
\end{equation*}
In summary, and recalling that $\|\Phi\|_\infty = \|\varphi\|_\infty \leq \lambda_0^{-1}\|\psi\|_\infty$ (from lemma \eqref{lemma:Linfty_est}), we have obtained that 
\begin{equation*}
\begin{aligned}
|J_1|&\leq C\epsilon^2\big(\||V|^2\Delta_VU\|_{L^1} + \| V\cdot \nabla_{V}U\|_{L^1} \\
&\hspace{4em}+ \| |V|^2V\cdot \nabla_{X'}U\|_{L^1} +\||V|^2U\|_{L^1}+\|U\|_{L^1}\big)\|\psi\|_\infty.
\end{aligned}
\end{equation*}
All the factors involving integrals of $U$ are finite due to lemma \ref{lemma:U_integrals}.

We now estimate $J_2$. When $\sigma$ and $\lambda$ are constant, then $J_2=0$. Otherwise, we notice first due to our hypothesis \eqref{def:delta_source} 
we can write $\eta = N - 2\epsilon^2\Theta$ for some $\Theta\in\RR^n$, $|\Theta|=O(1)$,  thus by Taylor expansion and the above we can write
\begin{equation*}
\begin{aligned}
J_2(\Phi) &= -\frac{1}{2}\epsilon\int_{\mathcal{Q}^\epsilon}\big(X'+\epsilon X^n\Theta',\epsilon X^n\Theta^n\big)\cdot\nabla \sigma(X^n\eta)(\Delta_VU)\Phi dXdV\\
&\hspace{3em}+2\epsilon \int_{\mathcal{Q}^\epsilon}\big(X'+\epsilon X^n\Theta',\epsilon X^n\Theta^n\big)\cdot\nabla\lambda(X^n\eta)U \Phi dXdV+ \|\Phi\|_\infty O(\epsilon^2)\\
\end{aligned}
\end{equation*}
where $\|\Phi\|_\infty\leq \lambda_0^{-1}\|\psi\|_\infty$ and with the remainder depending on integrals of the form
\begin{equation*}
\begin{aligned}
&\|\sigma\|_{C^2}\|\big(|X'+\epsilon X^n\Theta'|^2+\epsilon^2|X^n\Theta^n|^2\big)\Delta_VU\|_{L^1}\\
&\text{and}\quad \|\lambda\|_{C^2}\|\big(|X'+\epsilon X^n\Theta'|^2+\epsilon^2|X^n\Theta^n|^2\big)U\|_{L^1}
\end{aligned}
\end{equation*}
which are uniformly bounded for $\Theta$ in a bounded region (guaranteed by \eqref{def:delta_source}).

To gain the extra factor $\epsilon$ on the leading term, we notice that due to the symmetry of $U$ with respect to the spatial direction $\mathcal{S}(\eta) = \textstyle\frac{2}{1+\eta^n}\epsilon^2\Theta'$, for any vector $\xi\in \RR^{n-1}$, 
\begin{equation*}
\int_{\RR^{n-1}\times\RR^{n-1}}((X'+\epsilon X^n\Theta')\cdot\xi)UdX'dV=O(
\epsilon^2)
\end{equation*}
(this can be obtain for instance by passing to Fourier domain) while 
\begin{equation*}
\int_{\RR^{n-1}\times\RR^{n-1}}((X'+\epsilon X^n\Theta')\cdot\xi)\Delta_VUdX'dV=0.
\end{equation*}
Therefore we write
\begin{equation*}
\begin{aligned}
J_2&= -\frac{1}{4}\epsilon\int_{\RR^n_+\times\RR^{n-1}}(X'+\epsilon X^n\Theta')\cdot\nabla_{X'}\sigma(X^n\eta)\\
&\hspace{7em}\times(\Delta_VU)(\Phi(X',X^n,V)-\Phi(-\epsilon X^n\Theta',X^n,0)))dXdV \\
&\quad+2\epsilon\int_{\RR^n_+\times\RR^{n-1}}(X'+\epsilon X^n\Theta')\cdot\nabla_{X'}\lambda(X^n\eta)\\
&\hspace{7em}\times U(\Phi(X',X^n,V)-\Phi(-\epsilon X^n\Theta',X^n,0))dXdV + O(\epsilon^2).
\end{aligned}
\end{equation*}
Ideally, we would like to say that $\Phi$ is a Lipschitz continuous function with constant bounded by the $W^{1,\infty}$-norm of $\Psi$, and this would give us the extra $\epsilon$ we need for the second order approximation. However, we only have shown H\"older continuity of $\Phi$. A way of bypassing this lack of regularity is with the aid of the following lemma that allows us to substitute $\Phi$ with an $O(\epsilon)$-approximation $W$ that is indeed Lipschitz continuous. Its demonstration can be found at the end of this proof.
\begin{lemma}\label{lemma:ord1_approx} 
Let $\Phi$, $\Psi$ and $U$ be as above, and consider 
$$f=(X'+\epsilon X^n\Theta')\cdot\xi(X^n) U\quad\text{or}\quad f=(X'+\epsilon X^n\Theta')\cdot\xi(X^n)\Delta_VU,$$ 
for some bounded function $\xi$. Then, for $W$ solution to the backward Fermi pencil-beam system \eqref{backward_FPB}-\eqref{backward_FPB_bc}, there exists $C>0$, depending on $\sigma$, $\lambda$ and $\xi$, such that
\begin{equation*}
\Big|\int_{\mathcal{Q}} f\Phi dXdV - \int_{\mathcal{Q}} f WdXdV \Big|\leq C\epsilon\|\Psi\|_\infty.
\end{equation*}
\end{lemma}
Using this lemma, we replace $\Phi$ with $W$ in the last expression of $J_2$ to get
\begin{equation*}
\begin{aligned}
J_2&= -\frac{1}{4}\epsilon\int_{\RR^n_+\times\RR^{n-1}}(X'+\epsilon X^n\Theta')\cdot\nabla_{X'}\sigma(X^n\eta)\\
&\hspace{7em}\times(\Delta_VU)(W(X',X^n,V)-W(-\epsilon X^n\Theta',X^n,0)))dXdV \\
&\quad+2\epsilon\int_{\RR^n_+\times\RR^{n-1}}(X' +\epsilon X^n\Theta')\cdot\nabla_{X'}\lambda( X^n\eta)\\
&\hspace{7em}\times U(W(X',X^n,V)-W( -\epsilon X^n\Theta',X^n,0))dXdV + O(\epsilon^2),
\end{aligned}
\end{equation*}
which we rewrite more concisely as
\begin{equation*}
J_2 = \epsilon\int_{\mathcal{Q}} h(X,V)(W(X',X^n,V)-W( -\epsilon X^n\Theta',X^n,0))dXdV + O(\epsilon^2).
\end{equation*}
Lemma \ref{lemma:FPb_Lip} implies
\begin{equation*}
|J_2|\leq C\epsilon\||(X'+\epsilon X^n\Theta',V)|h(X',V)\|_{L^1(\mathcal{Q})}\sup_{X^n>0}\Lip_{X',V}(\Psi(X^n)) + O(\epsilon^2),
\end{equation*}
where the finiteness of the integral involving $h$ follows from lemma \ref{lemma:U_integrals}. Moreover, it follows directly from \eqref{def:Phi_Psi} that
\begin{equation*}
\Lip_{X',V}(\Psi(X^n)) \leq C\epsilon \Lip_{x',v}(\psi(X^n))\leq C\epsilon\kappa,
\end{equation*}
and thus we deduce the estimate
\begin{equation*}
|J_2|\leq C\epsilon^2(\|\psi\|_{\infty}+\kappa),
\end{equation*}
for a constant $C$ independent of $\Omega$, with the latter present in the estimate only through the support of $\psi$. The above estimate then depends only on the supremum norm and the Lipschitz constant of $\psi$, which are both uniformly bounded in $BL_{1,\kappa}$, thus we can take $\Omega$ arbitrary large. By taking supremum among all $\psi\in BL_{1,\kappa}(\Omega\times\SS^{n-1})$ we deduce that 
\begin{equation*}
\mathcal{W}^1_{\kappa,\Omega}\leq C\epsilon^2\kappa+o(1),
\end{equation*}
and the proof is completed by letting $\Omega \to \bar{\RR}^n_+$.
\end{proof}

\begin{proof}[Proof of lemma \ref{lemma:ord1_approx}]
Let $\bfU$ be a solution to the (forward) Fermi pencil-beam problem
\begin{equation*}
- \tilde{\sigma}\Delta_V \bfU  + V \cdot\nabla_{X'} \bfU+ \partial_{X^n}\bfU  + \tilde{\lambda}\bfU= f,\quad (X,V)\in \RR^{n}_+\times\RR^{n-1},\quad \bfU|_{X^n=0} = 0,
\end{equation*}
given explicitly by
\begin{equation*}
\bfU(X,V) = \int^{X^n}_{0}e^{-\int^{X^n}_t\tilde{\lambda}(s)ds}\tau_{X^nV}({\bf H}_2(t)*\tau_{-tV}f)dt.
\end{equation*}
Denoting 
\begin{equation*}
\bfu(x,\theta) = \frac{1}{(2\epsilon)^{2(n-1)}}\bfU((2\epsilon)^{-1}x',x^n,\epsilon^{-1}\mathcal{S}(\theta)),
\end{equation*}
we have
\begin{equation*}
\begin{aligned}
\int_{\RR^{n}_+\times\RR^{n-1}} fWdXdV &= \int_{\mathcal{Q}} \bfU\Psi dXdV \\
&= \underbrace{\int_{Q} \bfu(x,\theta)\psi(x,\theta)dxd\theta}_{I_1} +  \underbrace{\int_{\mathcal{Q}} \bfU\Psi (2^{1-n}c^{n-1}(\epsilon V)-1)dXdV}_{I_2}.
\end{aligned}
\end{equation*}
It is clear from \eqref{aux1} and lemma \ref{lemma:U_integrals} that 
\begin{equation*}
|I_2|\leq \epsilon^2\||V|^2\bfU\|_{L^1}\|\Psi\|_\infty\leq C\epsilon^2\|\Psi\|_\infty.
\end{equation*}
On the other hand, we use that $\varphi$ is a strong solution of \eqref{backward_FP} and we integrate by parts to get
\begin{equation*}
\begin{aligned}
I_1&= \int_{\mathcal{Q}} f\Phi dXdV + J_0(\Phi) + J_1(\Phi) + J_2(\Phi),
\end{aligned}
\end{equation*}
where
\begin{equation}\label{J0}
\begin{aligned}
J_0(\Phi) &=  \int_{\Gamma_-}\bfu\varphi(\theta\cdot\nu(x))dS(x)d\theta=0,
\end{aligned}
\end{equation}
since $U=0$ on $X^n=0$, and $J_1$ and $J_2$ defined as in \eqref{J1} and \eqref{J2}, respectively.
Estimating $J_1$ as in the proof of the previous theorem and using one more time lemma \ref{lemma:U_integrals}, we obtain
\begin{equation*}
\begin{aligned}
|J_1|&\leq C\epsilon^2\big(\||V|^2\Delta_V\bfU\|_{L^1} + \| V\cdot \nabla_{V}\bfU\|_{L^1} \\
&\hspace{3em}+ \| |V|^2V\cdot \nabla_{X'}\bfU\|_{L^1} +\|(1+|V|^2)\bfU\|_{L^1} + \||V|^2f\|_{L^1}\big)\|\Psi\|_\infty\leq C\epsilon^2\|\Psi\|_\infty,
\end{aligned}
\end{equation*}
while $|J_2|$ has a straightforward upper bound given by
$$|J_2|\leq C\epsilon\|\Phi\|_\infty(\||X'+\epsilon X^n\Theta'|\Delta_V\bfU\|_{L^1}+\||X'+\epsilon X^n\Theta'|\bfU\|_{L^1})\leq C\epsilon\|\Psi\|_\infty,$$
for a constant depending on $\|\sigma\|_{C^1}$ and $\|\lambda\|_{C^1}$.
\end{proof}

\subsubsection{$L^1$-boundary sources}\label{subsec:ssbs}
We now consider more general boundary sources such that for some $C>0$ independent of $\epsilon$,
\begin{equation}\label{bdry_cond_app2}
g(x,\theta) \in L^1(\Gamma_-), \quad\supp(g)\Subset \partial\RR^n_+\times\{\theta\in\SS^{n-1}:|\theta-N|<C\epsilon^2 \} \subset\Gamma_-.
\end{equation}

Let $\{U_{y,\eta}(X,V)\}_{y,\eta}$ be a family of pencil-beams with respective boundary conditions $G = \delta(X')\delta(V-\epsilon^{-1}\mathcal{S}(\eta))$ on $\partial\RR^n_+$ and null interior source $F=0$. The subscript $(y,\eta)\in \supp(g)$ indicates that each pencil beam is constructed using the coefficients
\begin{equation*}
\tilde{\sigma}(X^n) = \frac{1}{4}\sigma(y+X^n\eta)\quad\text{and}\quad \tilde{\lambda}(X^n) = \lambda(y+X^n \eta),
\end{equation*}
and they are all given by the explicit formula \eqref{FPb_sol_form}. 

We define $\mathfrak{u}(x,\theta;y,\eta)$ as the following transformation of $U_{y,\eta}$:
\begin{equation}\label{pb_y}
\mathfrak{u}(\cdot,\cdot;y,\eta) := (2\epsilon)^{-2(n-1)}U_{y,\eta}\circ T_{y}^\epsilon,\quad (y,\eta)\in \supp(g),
\end{equation}
where 
\begin{equation*}
T^\epsilon_{y}:\RR^n\times\SS^{n-1}\ni (x,\theta)\mapsto (X,V)\in \RR^n_+\times\RR^{n-1}
\end{equation*}
is the transformation defined as
\begin{equation*}
(X,V) = (( 2\epsilon)^{-1}(x'-y'),x^n,\epsilon^{-1}\mathcal{S}(\theta)).
\end{equation*}
In other words, $T^\epsilon_{y}$ defines stretched coordinates after performing a spatial translation $x\mapsto x-y$.
The rescaling is chosen such that it preserves (up to a $O(1)$ factor) the $L^1$ norm of the pencil-beam. Our superposition of pencil-beam model is the distribution in $\bar{Q}$ given by
\begin{equation}\label{sup_pb}
\mathfrak{u}(x,\theta) := \int_{\Gamma_-} g(y,\eta)\mathfrak{u}(x,\theta;y,\eta) dS(y)d\eta,
\end{equation}
so that
\begin{equation*}
\langle\mathfrak{u},\phi\rangle := \int_{\Gamma_-} \int_{Q}g(y,\eta)\mathfrak{u}(x,\theta;y,\eta)\phi(x,\theta)dxd\theta dS(y)d\eta,\quad\forall \phi\in C^\infty_0(\bar{Q}).
\end{equation*}
We can extend the definition of $\mathfrak{u}$ to all continuous functions in $\bar{Q}$ since for all $(z,\zeta)\in\supp(g)$, 
\begin{equation*}
\|\mathfrak{u}(\cdot,\cdot;z,\zeta)\|_{L^1}\leq \sup_{(y,\eta)\in\supp(g)}\|U_{y,\eta}\|_{L^1(\RR^n_+\times\RR^{n-1})}<+\infty.
\end{equation*}
In this setting, we have the following approximation result:

\begin{cor}\label{cor:main2}
Let $u$ be a distributional solution to \eqref{FP} with incoming boundary condition $g$ as in \eqref{bdry_cond_app2}. There exists a constant $C=C(\|\sigma\|_{C^3},\|\lambda\|_{C^2})>0$, independent of $\epsilon$, so that 
\begin{equation*}
\mathcal{W}^1_{\kappa}(u,\mathfrak{u})\leq C\|g\|_{L^1}\epsilon^2\kappa.
\end{equation*}
\end{cor}

\begin{proof} This proof is quite similar to that in the single pencil-beam case. We take an open and bounded $\Omega\subset\bar{\RR}^n_+$. For any $\psi\in BL_{1,\kappa}(\Omega\times\SS^{n-1})$, which we extend as a continuous compactly supported function preserving its Lipschitz constant, let $\varphi$ be the unique solution to the backward Fokker-Planck equation,
\begin{equation*}
 -\epsilon^2\sigma\Delta_\theta\varphi- \theta\cdot\nabla_x\varphi + \lambda\varphi = \psi,\quad \varphi|_{\Gamma_+} = 0.
 \end{equation*}
We have
\begin{equation*}
\begin{aligned}
\int_Q\psi(x,\theta) (u-\mathfrak{u})dxd\theta &=  \langle g|\theta\cdot\nu(x)|,\varphi\rangle_{\Gamma_-} \\
&\hspace{2em}-\int_{\Gamma_-} \int_{Q}g(y,\eta)\mathfrak{u}(x,\theta;y,\eta)\psi(x,\theta)dxd\theta dS(y)d\eta\\
&= \int_{\Gamma_-}g(y,\eta)\varphi(y,\eta)|\eta\cdot\nu(y)|dS(y)d\eta \\
&\hspace{2em}-\int_{\Gamma_-}g(y,\eta)\int_{Q}\mathfrak{u}(x,\theta;y,\eta)\psi(x,\theta)dxd\theta dS(y)d\eta. 
\end{aligned}
\end{equation*}
The proof reduces to showing that
\begin{equation}\label{plan_proof2a}
\Big|\int_{Q}\mathfrak{u}(x,\theta;y,\eta)\psi(x,\theta)dxd\theta -\varphi(y,\eta)\eta^n\Big|\leq C\epsilon^2\kappa,
\end{equation}
uniformly for $(y,\eta)\in\supp(g)$. 
Notice that applying the translation $z=x-y$ we get
\begin{equation*}
\begin{aligned}
\int_{Q}\mathfrak{u}(x,\theta;y,\eta)\psi(x,\theta)dxd\theta 
&= (\sqrt{2}\epsilon)^{-2(n-1)}\int_{Q}U_{y,\eta}(\epsilon^{-1}z',z^n,\epsilon^{-1}\mathcal{S}(\theta))\tilde{\psi}(z,\theta)dzd\theta,
\end{aligned}
\end{equation*}
with $\tilde{\psi}(z,\theta) = \psi(z+y,\theta)$. Thus, \eqref{plan_proof2a} is precisely what we obtained in the proof of Theorem \ref{thm:main1} (see \eqref{plan_proof1}). Going over those computations, one realizes that the constant in the estimate are uniform in $(y,\eta)$ and only depend on $\|\sigma\|_{C^3}$, $\|\lambda\|_{C^2}$ and the dimension $n$. 
We then have
\begin{equation*}
\int_{\Omega\times\SS^{n-1}}\psi(x,\theta) (u-\mathfrak{u})dxd\theta \leq C\epsilon^2\kappa\|g\|_{L^1} +o(1),
\end{equation*}
for all $\psi\in BL_{1,\kappa}(\Omega\times\SS^{n-1})$, for a constant $C>0$ independent of the set $\Omega$, thus we conclude by taking supremum among all those functions and then letting $\Omega\to\bar{\RR}^n_+$.
\end{proof}

\subsection{Comparison with ballistic linear transport}

We now compare the Fokker-Planck solution to the ballistic transport equation in the narrow-beam regime. Narrow beams are by essence well approximated by a distribution supported on a half line, at least when one considers approximations in the 1-Wasserstein distance. We consider such an approximation and show that the Fermi pencil-beam solution is significantly more accurate than the ballistic transport solution. At the appropriate beam scaling, i.e., at distances from the beam center scaled in $\epsilon$, we show (using $\kappa\approx\epsilon^{-1}$) that the Fermi pencil-beam is still accurate while the ballistic transport solution (obviously) fails to account for dispersion. 

Since diffusion is weak in a narrow beam regime, simply neglecting its effects leads to the ballistic transport model:
\begin{equation}\label{lin_transp}
\theta\cdot\nabla_x v + \lambda v = 0,\quad (x,\theta)\in\RR^n_+\times\SS^{n-1},\quad\text{with}\quad v|_{\Gamma_-} = g.
\end{equation}

Rather than comparing $v$ with $u$, we compare $v$ with $\mathfrak{u}$ instead since we already know how close $\mathfrak{u}$ is to $u$. We obtain the following result.
\begin{lemma}\label{lemma:upp_bound}
For $\sigma\in C^3(\bar{\RR}^n_+)$ and $\lambda\in C(\bar{\RR}^n_+)$, let $g$ and $\mathfrak{u}$ be as in \S\ref{subsec:ssbs}, and let $v$ be the solution to \eqref{lin_transp}. There exists a constant $C(\|\sigma\|_{C^3},\lambda_0)>0$ independent of $\epsilon$ such that
\begin{equation*}
\mathcal{W}^1_{\kappa}(v,\mathfrak{u})\leq C\kappa\epsilon\|g\|_{L^1}.
\end{equation*}
\end{lemma}

\begin{proof}
The transport solution $v$ can be written as
\begin{equation*}
v(x,\theta) = e^{-\int^{\tau_-(x,\theta)}_0\lambda(x-s\theta)ds}g(x-\tau_-(x,\theta)\theta,\theta).
\end{equation*}
Thus, applying the change of variables 
\begin{equation*}
(x,\theta) = \mathcal{T}(y',t,\eta) := ((y'+t\eta',t\eta^n),\eta),\quad |\nabla\mathcal{T}(y',t,\eta)|=|\eta^n|
\end{equation*}
we get that for any $\psi\in BL_{1,\kappa}(\bar{\RR}^n_+\times\SS^{n-1})$, 
\begin{equation*}
\begin{aligned}
\int v\psi dxd\theta 
&=\int_{\SS^{n-1}}\int_{\RR^{n-1}}\int_0^{\infty} g((y',0),\eta)e^{-\int^{t}_0\lambda((y',0)+s\eta)ds}\psi((y',0)+t\eta,\eta)\eta^ndtdy'd\eta\\
&=\int_{\SS^{n-1}}\int_{\RR^{n-1}}\int_0^{\infty} g((y',0),\eta)e^{-\int^{t}_0\lambda((y',0)+s\eta)ds}\psi((y',0)+t\eta,\eta)dtdy'd\eta\\
&\hspace{1em} + O(\epsilon^2\|g\|_{L^1}\|\psi\|_\infty),
\end{aligned}
\end{equation*}
where the last equality is obtained from the fact that for all $x\in\partial\RR^n_+$, $g(x,\cdot)$ is supported in a $\epsilon^2$-neighborhood of $N$.
On the other hand, recalling the definition of $\mathfrak{u}$ in \eqref{pb_y}-\eqref{sup_pb} and denoting $\Psi(X,V) = \psi(2\epsilon X',X^n,\mathcal{J}(\epsilon V))$ with $\mathcal{J}$ the inverse of the stereographic projection with respect to the south pole, we have that
\begin{equation*}
\begin{aligned}
&\int_Q \mathfrak{u}\psi dxd\theta \\
&\hspace{0em}= \int_{\Gamma_-\times\RR^{n-1}\times\RR^{n-1}_+} g((y',0),\eta)U_{y,\eta}(X'-{\textstyle \frac{1}{2\epsilon}}y,X^n,V)\Psi(X,V)\frac{c^{n-1}(\epsilon V)}{2^{n-1}}dXdVdy'd\eta\\
&\hspace{0em}=\int_{\Gamma_-\times \RR^{n-1}\times\RR^{n-1}_+} g((y',0),\eta)U_{y,\eta}(X'-{\textstyle \frac{1}{2\epsilon}}y',X^n,V)\Psi(X,V)dXdVdy'd\eta + R(\epsilon^2),
\end{aligned}
\end{equation*}
with remainder $|R(\epsilon^2)|\leq C\epsilon^2\|\psi\|_\infty\|g\|_{L^1}$. Combining both integrals above and performing the dilation $Y'={\textstyle \frac{1}{2\epsilon}}y'$ we obtain
\begin{equation}\label{v-u}
\begin{aligned}
&\int_Q\psi(x,\theta)(v-\mathfrak{u})dxd\theta\\
&\hspace{2em}=  \int_{\RR^{2(n-1)}} \int^{\infty}_0\tilde{g}(Y',W)e^{-\int^{X^n}_0\lambda(( 2\epsilon Y',0)+s\mathcal{J}(\epsilon W))ds}\\
&\hspace{11em}\times\Psi(Y',X^n,W)dX^n dY'dW\\
&\hspace{3em}-\int_{\RR^{2(n-1)}}  \int^{\infty}_0\tilde{g}(Y',W)\Big(\int_{\RR^{2(n-1)}} U_{(2\epsilon Y',0),\mathcal{J}(\epsilon W)}(X',X^n,V)\\
&\hspace{15em}\times\Psi(X'+Y,X^n,V)dX'dV\Big)dX^ndY' dW \\
&\hspace{3em}+ R(\epsilon^2),
\end{aligned}
\end{equation}
where $\tilde{g}(Y',W) := (2\epsilon)^{2(n-1)}g(( 2\epsilon Y',0),\mathcal{J}(\epsilon W))$, thus its $L^1$-norm is of the order of $\|g\|_{L^1}$. We recall that $U_{y,\eta}$ is the solution to the Fermi pencil-beam equation for $\tilde{\sigma}(X^n) = \frac{1}{4}\sigma(y+X^n\eta)$ and $\tilde{\lambda}(X^n)=\lambda(y+X^n\eta)$, and moreover
\begin{equation*}
\int_{\RR^{2(n-1)}} U_{y,\eta}(X',X^n,V)dX'dV = e^{-\int^{X^n}_0\lambda(y+s\eta)ds}.
\end{equation*}
Consequently,
\begin{equation*}
\begin{aligned}
&\int_Q\psi(x,\theta)(v-\mathfrak{u})dxd\theta\\
&\hspace{0em}\leq \int_{\RR^{2(n-1)}}  \int^{\infty}_0\tilde{g}(Y',W)\Big|\int_{\RR^{2(n-1)}} U_{(2\epsilon Y',0),\mathcal{J}(\epsilon W)}(X',X^n,V)\Psi(Y',X^n,W) dX'dV\\
&\hspace{3em}-  \int_{\RR^{2(n-1)}} U_{(2\epsilon Y',0),\mathcal{J}(\epsilon W)}(X',X^n,V)\Psi(X'+Y',X^n,V)dX'dV\Big|dX^ndY'dW\\
&\hspace{0em}\leq \int_{\RR^{2(n-1)}}  \int^{\infty}_0\tilde{g}(Y',W)\int_{\RR^{2(n-1)}} U_{(2\epsilon Y',0),\mathcal{J}(\epsilon W)}(X',X^n,V)\\
&\hspace{10em}\times\big|\Psi(Y',X^n,W) - \Psi(X'+Y',X^n,V)\big|dX'dVdX^ndY'dW\\
&\hspace{0em}\leq C\epsilon \Lip(\psi)\Big(\sup_{(y.\eta)\in\supp(g)}\|U_{y,\eta}(X,V)|(X',V)|\|_{L^1(\mathcal{Q})}\Big)\|\tilde{g}\|_{L^1}.
\end{aligned}
\end{equation*}
We conclude by noticing that $\|\tilde{g}\|_{L^1} = \|g\|_{L^1}$ and
\begin{equation*}
\sup_{(y,\eta)\in\supp(g)}\|U_{y,\eta}(X,V)|(X',V)|\|_{L^1(\mathcal{Q})}< \infty,
\end{equation*}
and taking supremum over all $\psi\in BL_{1,\kappa}(\bar{\RR}^n_+\times\SS^{n-1})$.
\end{proof}

We now obtain a lower bound on the mismatch between ballistic transport and the pencil-beam model. This shows that spreading is indeed not accounted for by the ballistic transport solution, and this effect becomes visible once we set a high measuring resolution of order $\epsilon^{-1}$ or higher.
\begin{lemma}\label{lemma:lower_bound}
For a resolution parameter $\kappa\gtrsim\epsilon^{-1}$ and $\sigma,\lambda\in C(\bar{\RR}^n_+)$, there exists a constant $C(n,\sigma_0,\|\lambda\|_\infty)>0$ such that
\begin{equation*}
\frac{1}{C}\|g\|_{L^1}\leq \mathcal{W}^1_{\kappa}(v,\mathfrak{u}).
\end{equation*}
\end{lemma}

\begin{proof}
For any $\psi\in BL_{1,\kappa}(\bar{\RR}^n_+\times\SS^{n-1})$, we have equality \eqref{v-u} for the rescaling $\Psi(X,V) = \psi( 2\epsilon X',X^n,\mathcal{J}(\epsilon V))$. We take an specific 
$\psi(x,\theta) = \psi_1(x^n)\psi_2(\theta)$ satisfying that:
\begin{itemize}
\item[1.] $0\leq \psi_1,\psi_2\leq 1$ and $\|\psi\|_\infty =1$;
\item[2.] $\psi_2(\theta)=1$ in the support of $g$ (i.e. for all $|N-\theta|\leq C\epsilon^2$);
\item[3.] denoting $\Psi_2(V) = \psi_2(\mathcal{J}(\epsilon V))$, $\supp\Psi_2\subset B(\eta)$ with $\eta>0$ to be chosen; 
\item[4.] $\supp\psi_1\subset (a,b)\Subset(0,\infty)$.
\end{itemize} 
For this choice of test function (and recalling $\tilde{g}(Y',W) = (2\epsilon)^{2(n-1)}g(2\epsilon Y',0,\mathcal{J}(\epsilon W))$) we have
\begin{equation*}
\begin{aligned}
&\int \psi(x,\theta)(v-\mathfrak{u})dxd\theta\\
&\hspace{0em}\geq  \int_{\RR^{2(n-1)}} \int^{\infty}_0\tilde{g}(Y',W)e^{-\int^{X^n}_0\lambda((\epsilon Y',0)+s(\epsilon W,1))ds}\\
&\hspace{6em}\times\Psi_1(X^n)\Psi_2(W)dX^n dY'dW \\
&\hspace{1em}-\int_{\RR^{2(n-1)}}  \int^{\infty}_0\tilde{g}(Y',W)\Psi_1(X^n)\\
&\hspace{6em}\times\Big(\int_{\RR^{2(n-1)}} U_{(\epsilon Y',0),\mathcal{J}(\epsilon W)}(X',X^n,V)\Psi_2(V)dX'dV\Big)dX^ndY'dW  \\
&\hspace{1em} - |R(\epsilon^2)|.\\
\end{aligned}
\end{equation*}
We rewrite this inequality as
\begin{equation}\label{lower_bound_psi}
\begin{aligned}
\int \psi(x,\theta)(v-\mathfrak{u})dxd\theta&\geq \int_{\RR^{2(n-1)}} \int^{\infty}_0\tilde{g}(Y',W)e^{-\int^{X^n}_0\lambda((\epsilon Y',0)+s(\epsilon W,1))ds}\\
&\hspace{2em}\times I_\eta(X^n;Y',W)\Psi_1(X^n)dX^ndY'dW - |R(\epsilon^2)|,
\end{aligned}
\end{equation}
where $I_\eta(X^n;Y',W)$ is defined as
\begin{equation*}
\begin{aligned}
&I_\eta(X^n;Y',W) \\
&\hspace{1em}= 1 - \int_{\RR^{n-1}\times B(\eta)}e^{\int^{X^n}_0\lambda((\epsilon Y',0)+s(\epsilon W,1))ds} U_{(\epsilon Y',0),\mathcal{J}(\epsilon W)}(X',X^n,V) dX'dV.
\end{aligned}
\end{equation*}
We see that
\begin{equation*}
\begin{aligned}
I_\eta(X^n;Y',W)&= 1- \int_{B(\eta)} \frac{e^{-\frac{|V|^2}{4E_0(X^n;Y',W)}}}{(4\pi E_0(X^n;Y',W))^{\frac{n-1}{2}}}dV\\ 
&= \int_{|V|>\eta} \frac{e^{-\frac{|V|^2}{4E_0(X^n;Y',W)}}}{(4\pi E_0(X^n;Y',W))^{\frac{n-1}{2}}}dV,
\end{aligned}
\end{equation*}
with $E_0(X^n;Y',W) = 4\int^{X^n}_0\sigma((\epsilon Y',0)+s(\epsilon W,1))dt$. By considering spherical coordinates and noticing that
\begin{equation*}
\frac{\eta}{\sqrt{4E_0(X^n;Y',W)}}\leq \frac{\eta}{4\sqrt{a\sigma_0}}=:\eta_0,\quad\forall (Y',W)\in\supp(\tilde{g}),\;X^n\in(a,b),
\end{equation*}
we have
\begin{equation*}
I_\eta(X^n;Y',W) \geq c_n\int^\infty_{\eta_0}e^{-r^2}r^{n-2}dr,
\end{equation*}
for $c_n = \frac{|\SS^{n-2}|}{\pi^{(n-1)/2}} = 2\Gamma((n-1)/2)^{-1}$. For $n=2$, we can choose $R>\eta^0$ so that
\begin{equation}\label{case_n2}
\int_{\eta_0}^\infty e^{-r^2}dr\geq \frac{1}{R}\int^R_{\eta_0}e^{-r^2}rdr 
= \frac{1}{2R}(e^{-\eta_0^2}-e^{-R^2}).
\end{equation}
If now $n\geq 3$, it is not hard to verify that
\begin{equation}\label{case_ngeq3}
\int_{\eta_0}^\infty e^{-r^2}r^{n-2}dr\geq e^{-\eta_0^2}\sum^{n-3}_{k=0}\frac{k!}{2^k}\eta_0^{n-3-k}.
\end{equation}

The parameter $\eta$ imposes a restriction on the size of the support of the (angular) test function $\Psi_2(V)$, which says that $\supp\Psi_2\subset B(\eta)$. In order to fulfill this support condition we want our test function not to decay too slowly, which means we need to have a function $\Psi$ as above and such that
\begin{equation*}
|\Psi_2\|_\infty=1\quad\text{and}\quad\text{Lip}_V(\Psi_2)\geq \eta^{-1}.
\end{equation*}
Going back to the original angular variable $\theta$ this translates into 
\begin{equation*}
C\text{Lip}_\theta(\psi_2)\geq (\epsilon\eta)^{-1},
\end{equation*}
for some constant $C>0$ (independent of the parameters), hence, this condition is fulfilled by choosing $\eta$ such that
\begin{equation*}
\kappa\gtrsim \epsilon^{-1}\eta^{-1}.
\end{equation*}
In dimension 2, we choose $\eta=O(\epsilon^{-1}\kappa^{-1})$ so that $\eta_0 = \epsilon^{-1}\kappa^{-1}$, and $R=\sqrt{2}\eta_0$, thus we obtain 
\begin{equation*}
I_\eta(X^n;Y',W) \geq c\epsilon\kappa e^{-{(\epsilon\kappa)}^{-2}}(1-e^{-{(\epsilon\kappa)}^{-2}}),
\end{equation*}
for some constant $c$ independent of $\epsilon$ and uniform with respect to $X^n,Y'$ and $W$. According to our assumption $\kappa\gtrsim\epsilon^{-1}$ we deduce that for some $c>0$,
\begin{equation*}
I_\eta(X^n;Y',W) \geq c\epsilon\kappa.
\end{equation*}
Similarly in dimension $n\geq 3$, for $\eta$ as above we get the next uniform lower bound 
\begin{equation*}
I_\eta(X^n;Y',W) \geq e^{-(\epsilon\kappa)^{-2}}\sum^{n-3}_{k=0}\frac{k!}{2^k}(\epsilon\kappa)^{-(n-3-k)}\geq c,
\end{equation*}
for another constant $c>0$. Introducing this lower bounds into \eqref{lower_bound_psi} yields that for $\kappa\gtrsim\epsilon^{-1}$ 
\begin{equation*}
\|\lambda\|_\infty^{-1}\|g\|_{L^1}(1 - 
\epsilon^2)\leq C\int \psi(x,\theta)(v-\mathfrak{u})dxd\theta\leq C\mathcal{W}^1_{\kappa}(v,\mathfrak{u}).
\end{equation*}
\end{proof}

We summarize the previous results of this section in the form of our main theorem which we rephrase here.\\

{\sc Theorem 1.1.} {\em Let $u$ be the solution to Fokker-Planck with boundary source $g$ as in \eqref{bdry_cond_app2} and $\|g\|_{L^1}=1$. Let $\mathfrak{u}$ the superposition of pencil beam in \eqref{sup_pb} and $v$ the ballistic transport solution of \eqref{lin_transp}. For any dimension $n\geq 2$ there exists a constant $C(n,\|\sigma\|_{C^3},\|\lambda\|_{C^2})>0$ such that for $\kappa\gtrsim 1$,
\begin{equation*}
\mathcal{W}^1_{\kappa}(u,v)\leq C\kappa\epsilon\quad\text{and}\quad\mathcal{W}^1_{\kappa}(u,\mathfrak{u})\leq C\epsilon^2\kappa.
\end{equation*}
Moreover, if the resolution parameter is such that $\kappa\approx\epsilon^{-1}$, then there is $C>0$ so that
\begin{equation*}
C^{-1}\leq \mathcal{W}^1_{\kappa}(u,v)\leq C \quad\text{and}\quad \mathcal{W}^1_{\kappa}(u,\mathfrak{u})=O(\epsilon).
\end{equation*}
}
\begin{proof}
It follows directly from corollary \ref{cor:main2} and lemmas \ref{lemma:upp_bound} and \ref{lemma:lower_bound}.
\end{proof}


\section*{Acknowledgement}
The authors would like to thank Luis Silvestre for multiple comments on the manuscript and for bringing \cite{IS} to their attention.
This research was partially supported by the Office of Naval Research, Grant N00014-17-1-2096 and by the National Science Foundation, Grant DMS-1908736. 
\begin{appendix}

\section{Complete proof of Theorem \ref{thm:higher_reg}}\label{appdx:proofThm_reg}
We follow the method presented in \cite{Bo} which is based on a H\"ormander-type identity for the commutator of certain operators.

\noindent {\it 1) Localization and Mollification.} 
Let $\eta^\delta$ be a mollifier and $\chi$ a cutoff function supported in a neighborhood $\mathcal{U}\Subset Q$ of a point $(x_0,\theta_0)$. Let $\mathcal{V}$ and $\mathcal{W}$ be two open sets containing $(x_0,\theta_0)$ and such that $\chi = 1$ in $\mathcal{V}$ and $\mathcal{V}\Subset\mathcal{U}\Subset\mathcal{W}$. We choose $\delta>0$ small enough so that $\mathcal{U}+\supp(\eta^\delta)\subset \mathcal{W}$. Since we can always rotate the coordinate system in $\RR^n$, we lose no generality in assuming that $\theta_0$ is contained in the span of $(1,0\dots,0)$ and $(0,\dots,0,1)=N$. We first consider beams coordinates on $\SS^{n-1}$ with respect to the north pole $N$. Then $N$ is identified with $0\in \RR^{n-1}$ and we write $\eta^\delta(x,v)$ and $\chi(x,v)$ for the respective representative function of $\eta^\delta$ and $\chi$ in the local coordinates. We abuse the notation and write $\mathcal{U}$ interchangeably for the set considered above and its image under local coordinates.

Let's plug
$$\phi(y,\xi) = \eta^\delta(x-y,\theta-\xi)\chi(x,\theta)$$
into equation \eqref{FP_loc_coord}. From the transport part we get
$$
\begin{aligned}
&\hspace{-3em}-\chi(x,v)\int\ u(y,w)(\partial_{y^n}+w\cdot \nabla_{y'} )[\eta^\delta(x-y,v-w)]dyd w \\
&\hspace{1em}=\chi(x,v)\int\ u(y,w)(\partial_{x^n}\eta^\delta(x-y,v-w)+v\cdot \nabla_{x'} \eta^\delta(x-y,v-w))dyd w \\
&\hspace{8em}-\chi(x,v)\int\ u(y,w)((v-w)\cdot \nabla_{x'} \eta^\delta(x-y,v-w))dyd w\\
&\hspace{1em}=(\partial_{x^n}+v\cdot\nabla_{x'})[\chi (u*\eta^\delta)] - \underbrace{(u*\eta^\delta)(\partial_{x^n}\chi+v\cdot\nabla_{x'}\chi)+\chi(u*(v\cdot\nabla_{x'}\eta^\delta))}_{=:g_1}.
\end{aligned}
$$
Using the notation in \eqref{opAB} the diffusion term gives
$$
\begin{aligned}
&\chi(x,v)\int \langle\frac{\epsilon^2\sigma(y)}{\langle w\rangle^{n-2}}(Id + ww^T)\nabla_w u,\nabla_w(\langle w\rangle^{n+1}\eta^\delta(x-y,v-w))\rangle dydw\\
&\hspace{3em}=-\chi(x,v)\int \nabla_wu \cdot(A(y,w) \nabla_v\eta^\delta(x-y,v-w)) + u(B(x,w)\cdot \nabla_v\eta^\delta(x-y,y-w)) dydw\\
&\hspace{3em}=-\nabla_v\cdot(A\nabla_v(\chi(u*\eta^\delta)) + (u*\eta^\delta)(\nabla_v\cdot A\nabla_v\chi)\\
&\hspace{6em} + B\cdot\nabla_v(\chi(u*\eta^\delta)) - (u*\eta^\delta)(B\cdot\nabla_v\chi)\\
&\hspace{6em} - \chi(x,v)\int \nabla_wu \cdot H_A(y,w)[(x-y,v-w),\nabla_v\eta^\delta(x-y,v-w)]dydw \\
&\hspace{6em} - \chi(x,v)\int u \;H_B(y,w)[(x-y,v-w),\nabla_v\eta^\delta(x-y,y-w)] dydw\\
&\hspace{3em}=-\nabla_v\cdot(A\nabla_v(\chi(u*\eta^\delta)) + B\cdot\nabla_v(\chi(u*\eta^\delta)) - g_2,
\end{aligned}
$$
for some multi-linear operators $H_A$ and $H_B$ arising from the Taylor expansion of $A$ and $B$, respectively, around the point $(x,v)$. Finally, the zero order term in \eqref{FP_loc_coord} gives
$$
\begin{aligned}
&\chi(x,v)\int c(y,w)u(y,w)\eta^\delta(x-y,v-w)dydw \\
&\hspace{3em}= c(\chi(u*\eta^\delta)) - \underbrace{\chi(x,v)\int u(y,w)((x-y,v-x)\cdot h_c(y,w))\eta^\delta(x-y,v-w)dydx}_{=:g_3}\\
\end{aligned}
$$
with $h_c$ from the remainder of the first order Taylor expansion of $c(x,v)$. 
From the previous we deduce that the function
$$u^\delta := (u*\eta^\delta)\chi\in C^\infty_c(\mathcal{U}),$$ 
satisfies
\begin{equation}\label{smooth_FP}
-\nabla_v\cdot(A\nabla_vu^\delta) + \partial_{x^n}u^\delta+v\cdot\nabla_{x'} u^\delta + B\cdot\nabla_vu^\delta+ c u^\delta = f^\delta,\quad\forall (x,v)\in \mathcal{U},
\end{equation}
with 
$$f^\delta:= (\hat{f}*\eta^\delta)\chi + g_1+g_2+g_3\in C_c^\infty(\mathcal{U}).$$
One verifies that
$$
\|g_1+g_2+g_3\|_{L^2(\mathcal{U})}\leq C(\| u\|_{L^2(\mathcal{U})} + \epsilon^2\|\nabla_v u\|_{L^2(\mathcal{U})}),
$$
for some constant depending on $\|\sigma\|_{C^1}$, $\|\lambda\|_{C^1}$, $\|\eta^\delta\|_{L^1}$ and $\||(x,\theta)|\nabla_\theta\eta^\delta\|_{L^1}$.

Since $u, \nabla_v u\in L^2(\mathcal{V})$ and $\chi=1$ in $\mathcal{V}$, then $u^\delta\to u$ and $\nabla_v u^\delta\to \nabla_v u$, as $\delta\to 0$, in the $L^2$-sense. Furthermore we have the following estimate:
\begin{equation}\label{ub_fsmooth}
\begin{aligned}
\|f^\delta\|_{L^2(\mathcal{U})} &\leq \|f\|_{L^2(\mathcal{W})} + C(\epsilon^2 \|\nabla_\theta u\|_{L^2(\mathcal{W})}+\|u\|_{L^2(\mathcal{W})})\\
&\leq \|f\|_{L^2(\mathcal{W})} +C(\epsilon^2 \|\nabla_\theta u\|_{L^2(Q)}+\|u\|_{L^2(Q)}),
\end{aligned}
\end{equation}
uniformly with respect to $\delta\ll 1$, and where the constant $C>0$ depends on $\sigma$ and $\lambda$ through they $C^1$-norms.\\

\noindent{\it 2) Higher regularity estimates for smooth compactly supported solutions.}  Let's drop the subindex $\delta$ for a moment and assume $u$ is smooth and compactly supported.
Without lost of generality we assume $\mathcal{U} = X_0\times\Theta_0$, neighborhood of $(x_0,\theta_0)\in Q$. 

At the core of this proof is the next commutator identity. Denoting the transport operator 
$$T = (\partial_{x^n} + v\cdot\nabla_{x'}),$$
we have: 
$$\partial_{x^j} = \partial_{v_j}T - T\partial_{v^j},\quad j=1,\dots,n-1.$$
Under our new notation \eqref{smooth_FP} rewrites as
\begin{equation}\label{eq_T}
\mathcal{L}u:=T u -\nabla_v\cdot(A\nabla_vu) + B\cdot\nabla_vu+ c u = f,\quad\forall (x,v)\in \mathcal{U}.
\end{equation}
We abbreviate the above expression as $ Tu = h$ with $h$ of the form
$$h = \nabla_v\cdot(A\nabla_vu) - B\cdot\nabla_vu- c u+f.$$
Denoting $\langle f,g\rangle = \int fg dxdv$ the $L^2$-bracket, we have:
$$
\begin{aligned}
\|D^{-1/3}_{x'}\partial_{x^j}u\|^2_{L^2} &= \langle D^{-2/3}_{x'}\partial_{x^j}\bar{u},\partial_{x^j}u\rangle\\
&=\langle D^{-2/3}_{x'}\partial_{x^j}\bar{u},\partial_{v_j}Tu  - T\partial_{v^j}u \rangle\\
&=\langle D^{-2/3}_{x'}\partial_{x^j}\bar{u},\partial_{v_j}h  -T\partial_{v^j}u \rangle\\
&=-\langle \partial_{v^j}D^{-2/3}_{x'}\partial_{x^j}\bar{u},h \rangle + \langle  T D^{-2/3}_{x'}\partial_{x^j}\bar{u},\partial_{v^j}u \rangle\\
&=-\langle \partial_{v^j}D^{-2/3}_{x'}\partial_{x^j}\bar{u},h \rangle + \langle  D^{-2/3}_{x'}\partial_{x^j}T\bar{u},\partial_{v^j}u \rangle\\
&=-\langle \partial_{v^j}D^{-2/3}_{x'}\partial_{x^j}\bar{u},h \rangle - \langle  \bar{h},D^{-2/3}_{x'}\partial_{x^j}\partial_{v^j}u \rangle\\
&=-\langle \partial_{v^j}D^{-2/3}_{x'}\partial_{x^j}\bar{u},h \rangle - \langle  D^{-2/3}_{x'}\partial_{x^j}\partial_{v^j}u ,\bar{h}\rangle\\
&=-2\mathfrak{Re}\langle D^{-2/3}_{x'}\partial_{x^j}\partial_{v^j}\bar{u}, h \rangle.
\end{aligned}
$$
We can then bound from above as follows,
$$
\begin{aligned}
\|D^{-1/3}_{x'}\partial_{x^j}u\|^2_{L^2} &\leq 2\|D^{-2/3}_{x'}\partial_{x^j}\partial_{v^j}u\|_{L^2}\| h \|_{L^2}\\
&\leq C\sum_{i=1}^n\|D^{-2/3}_{x'}\partial_{x^i}\nabla_{v}u\|_{L^2}\| h \|_{L^2}\\
&\leq C\|D^{1/3}_{x'}\nabla_{v}u\|_{L^2}\| h \|_{L^2}.
\end{aligned}
$$
therefore,
\begin{equation}\label{ineq_2/3der}
\| D^{2/3}_{x'}u\|^2_{L^2}\leq C\|D^{1/3}_{x'}\nabla_v u\|_{L^2} \| h\|_{L^2}.
\end{equation}

On the other hand, 
$$
\mathcal{L}D^{1/3}_{x'}u = D^{1/3}_{x'}f + [\mathcal{L},D^{1/3}_{x'}]u,
$$
which implies
$$
(T - \nabla_v\cdot A\nabla_v)D^{1/3}_{x'}u = D^{1/3}_{x'}f +\nabla_v\cdot(D^{1/3}_{x'}A)\nabla_vu - D^{1/3}_{x'}(B\cdot \nabla_vu)- D^{1/3}_{x'}(cu).
$$
Multiplying by $D^{1/3}_{x'}\bar{u}$ and integrating gives us
$$
\begin{aligned}
\int_{\RR^{n}_+\times\RR^{n-1}} \langle A\nabla_v D^{1/3}_{x'}u,\nabla_v D^{1/3}_{x'}u\rangle dxdv 
&= \mathfrak{Re}\langle D^{1/3}_{x'}\bar{u},D^{1/3}_{x'}(f- c u)\rangle + \mathfrak{Re}\langle D^{1/3}_{x'}\bar{u}, \nabla_v\cdot (D^{1/3}_{x'}A)\nabla_v u\rangle\\
&= \mathfrak{Re}\langle D^{2/3}_{x'}\bar{u},(f- c u)\rangle-\mathfrak{Re}\langle D^{1/3}_{x'}\nabla_v\bar{u}, (D^{1/3}_{x'}A)\nabla_v u\rangle,
\end{aligned}
$$
where some of the term have vanished due to integration by parts. 
Then, recalling that $A$ is positive definite, we obtain
$$
\begin{aligned}
\epsilon^2\|\nabla_v D^{1/3}_{x'}u\|^2_{L^2} 
&\leq C\|D^{2/3}_{x'}u\|_{L^2}\|f\|_{L^2} + \|\nabla_v D^{1/3}_{x'}u\|_{L^2}\|(D^{1/3}_{x'}A)\nabla_v u\|_{L^2}\\
&\leq C\|D^{2/3}_{x'}u\|_{L^2}\|f\|_{L^2} + C\epsilon\|\nabla_v D^{1/3}_{x'}u\|_{L^2}\|f\|_{L^2},
\end{aligned}
$$
where we used the inequality $\epsilon\|\nabla_v u\| + \|u\|_{L^2}\leq \|f\|_{L^2}$ which is the natural energy estimate of \eqref{eq_T}.
We can directly combine the previous inequality with \eqref{ineq_2/3der} to deduce
$$
\begin{aligned}
\epsilon^2\|\nabla_v D^{1/3}_{x'}u\|^2_{L^2} \leq C\|D^{1/3}_x\nabla_v u\|_{L^2}^{1/2} \| h\|_{L^2}^{1/2}\|f\|_{L^2}  +  C\epsilon\|\nabla_v D^{1/3}_{x'}u\|_{L^2}\|f\|_{L^2} .
\end{aligned}
$$
Simplifying some terms and applying Young's inequality: $ab\leq\frac{a^p}{p}+\frac{b^q}{q}$ for $p^{-1}+q^{-1}=1$; we get
$$
\|\nabla_v D^{1/3}_{x'}u\|^{3/2}_{L^2} \leq C\epsilon^{-2}\| h\|_{L^2}^{1/2}\|f\|_{L^2}  + C\epsilon^{-3/2}\|f\|_{L^2}^{3/2},
$$
which can be rewritten as
\begin{equation*}\label{ineq_1/3der}
\|\nabla_v D^{1/3}_{x'}u\|_{L^2}\leq C\big(\epsilon^{-4/3}\| h\|_{L^2}^{1/3}\|f\|_{L^2}^{2/3}  + \epsilon^{-1}\|f\|_{L^2}\big).
\end{equation*}
Plugging this into \eqref{ineq_2/3der} gives
\begin{equation}\label{ineq_2/3der2}
\begin{aligned}
\| D^{2/3}_{x'}u \|_{L^2}&\leq C\big(\epsilon^{-2/3}\| h\|_{L^2}^{2/3}\|f\|_{L^2}^{1/3}  + \epsilon^{-1/2}\| f\|_{L^2}^{1/2}\|h\|_{L^2}^{1/2}\big).
\end{aligned}
\end{equation}
\smallskip

If we now multiply \eqref{eq_T} by $-\nabla_v\cdot A\nabla_v \bar{u}$ and integrate, this yields
$$\mathfrak{Re}\langle -\nabla_v\cdot A\nabla_v \bar{u},Tu \rangle + \int |-\nabla_v\cdot A\nabla_vu|^2dx dv = \mathfrak{Re}\langle -\nabla_v\cdot A\nabla_v \bar{u}, -B\cdot\nabla_vu-c u+f\rangle.$$
However,
$$
\begin{aligned}
\big|\mathfrak{Re}\langle-\nabla_v\cdot A\nabla_v \bar{u},Tu \rangle\big| &\leq \big|\mathfrak{Re}\langle A\nabla_v\bar{u},\nabla_vTu \rangle\big|\\
&=\big|\mathfrak{Re}\langle A\nabla_{v}\bar{u},T\nabla_{v}u +\nabla_{x'}u\rangle\big| \\
&= \big|\mathfrak{Re}\langle[-T,A]\nabla_{v}\bar{u},\nabla_vu\rangle \big|+\big|\mathfrak{Re}\langle A\nabla_{v}\bar{u},\nabla_{x'}u\rangle \big| \\
&\leq C(\epsilon^2\|\nabla_vu\|^2_{L^2} + \|D^{1/3}_{x'}\nabla_vu\|_{L^2}\|D^{-1/3}_{x'}\nabla_{x'}u\|_{L^2})\\
&\leq C(\epsilon^2\|\nabla_vu\|^2_{L^2} + \|D^{1/3}_{x'}\nabla_vu\|_{L^2}\|D^{2/3}_{x'}u\|_{L^2}).
\end{aligned}
$$
therefore
$$
\begin{aligned}
\|\nabla_v\cdot A\nabla_v u\|^2_{L^2} &\leq C\|\nabla_v\cdot A\nabla_vu \|_{L^2}\big(\epsilon^2\|\nabla_vu\|_{L^2} + \|u\|_{L^2}+\|f\|_{L^2}\big) \\
&\hspace{2em}+C(\epsilon^2\|\nabla_vu\|^2_{L^2} + \|D^{1/3}_{x'}\nabla_vu\|_{L^2}\|D^{2/3}_{x'}u\|_{L^2}).
\end{aligned}
$$
It follows from H\"older inequality in Fourier domain that
$$\|D^{1/3}_{x'}\nabla_v u\|_{L^2} \leq \|D^{2/3}_{x'}u\|^{1/2}_{L^2}\| \Delta
_v u\|_{L^2}^{1/2}\leq C\|D^{2/3}_{x'}u\|^{1/2}_{L^2}\| \nabla
_v\cdot A\nabla_v u\|_{L^2}^{1/2},$$
which plugged into the previous estimate gives
$$
\begin{aligned}
\|\nabla_v\cdot A\nabla_v u\|^2_{L^2}  &\leq C\|\nabla_v\cdot A\nabla_vu \|_{L^2}\big(\epsilon^2\|\nabla_vu\|_{L^2} + \|u\|_{L^2}+\|f\|_{L^2}\big) \\
&\hspace{3em}+C(\epsilon^2\|\nabla_vu\|^2_{L^2} + \| \nabla
_v\cdot A\nabla_v u\|_{L^2}^{1/2}\|D^{2/3}_{x'}u\|_{L^2}^{3/2})\\
&\leq  C\|\nabla_v\cdot A\nabla_vu \|_{L^2}\|f\|_{L^2}+C\|f\|_{L^2}^2\\
&\hspace{3em}+C \| \nabla
_v\cdot A\nabla_v u\|_{L^2}^{1/2}\big(\epsilon^{-1}\| h\|_{L^2}\|f\|_{L^2}^{1/2}  + \epsilon^{-3/4}\|f\|_{L^2}^{3/4}\|h\|_{L^2}^{3/4}\big),
\end{aligned}
$$
which then implies
$$
\begin{aligned}
\|\nabla_v\cdot A\nabla_v u\|^{2}_{L^2}  &\leq  C\|f\|_{L^2}^2+C\big(\epsilon^{-1}\| h\|_{L^2}\|f\|_{L^2}^{1/2}  + \epsilon^{-3/4}\|f\|_{L^2}^{3/4}\|h\|_{L^2}^{3/4}\big)^{4/3}\\
&\leq  C\big(\|f\|_{L^2}^2+\epsilon^{-4/3}\| h\|_{L^2}^{4/3}\|f\|_{L^2}^{2/3}  + \epsilon^{-1}\|f\|_{L^2}\|h\|_{L^2}\big)
\end{aligned}
$$
Young's inequality yields
$$
\|\nabla_v\cdot A\nabla_v u\|_{L^2}\leq C\big(\|f\|_{L^2}+\epsilon^{2}\|h\|\big),
$$
thus, from the estimate $\|h\|_{L^2}\leq \|\nabla_v\cdot A\nabla_v u\|_{L^2} + C\|f\|_{L^2}$, we finally deduce
$$
\|\nabla_v\cdot A\nabla_v u\|_{L^2}\leq C\|f\|_{L^2}.
$$
It follows from the local representation of $\Delta_\theta$ in beam coordinates that
$$
\epsilon^{2}\|\Delta_\theta u\|_{L^2}\leq C\|f\|_{L^2}.
$$
On the other hand, the previous inequality implies $\|h\|_{L^2}\leq C\|f\|_{L^2}$, therefore \eqref{ineq_2/3der2} yields
$$
\begin{aligned}
\| D^{2/3}_{x'}u \|_{L^2}&\leq C\big(\epsilon^{-2/3}\|f\|_{L^2}  +\epsilon^{-1/2}\|f\|_{L^2}\big)\leq C\epsilon^{-2/3}\|f\|_{L^2}.
\end{aligned}
$$

Considering local coordinates given instead by beam coordinates around $(1,0\dots,0)$ (whose domain contains $\theta_0$), we can repeat the computations above in terms of the pseudodifferential operator $D^s_{x''}= (1-\Delta_{x''})^{s/2}$, where we decompose the spatial variables as $x = (x^1,x'')$. Consequently, we obtain an analogous inequality for $\|D^{2/3}_{x''}u\|_{L^2}$. It then follows that 
$$\| D^{2/3}_{x}u \|_{L^2}\leq C\epsilon^{-2/3}\|f\|_{L^2},$$
for $D^s_{x}= (1-\Delta_{x})^{s/2}$, which is obtained by recalling that in Fourier space:
$$
(1+|\xi|^2)^{s/2}\leq(1+|\xi'|^2 + 1+|\xi''|^2)^{s/2}\leq (1+|\xi'|^2)^{s/2}+(1+|\xi''|^2)^{s/2}.
$$

\noindent{\it 3) Back to original solution.} We now go back to our original notation and write $u^\delta$ for the smooth compactly supported solution and $u$ the original solution to \eqref{FP}. From the previous, we can find a sequence $\{\delta_j\}$ so that $\Delta_\theta u^{\delta_j}$ and $D^{2/3}_xu^{\delta_j}$ converge weakly in $L^2(\mathcal{V})$ to $\Delta_\theta u,D^{2/3}_xu\in L^2(\mathcal{V})$, respectively, as $j\to\infty$. Moreover, 
$$\|D^{2/3}_xu\|_{L^2(\mathcal{V})}\leq \liminf\|D^{2/3}_xu^{\delta_j}\|_{L^2(\mathcal{V})}\leq C\|f^{\delta_j}\|_{L^2(\mathcal{U})}\leq C(\|f\|_{L^2(\mathcal{W})} + \|u\|_{\mathcal{H}}),$$
and similarly for $\|\Delta_\theta u\|_{L^2(\mathcal{V})}$.
We have proven that $\Delta_\theta u,D^{2/3}u\in L^2_{loc}(Q)$. 
Moreover, for every compact $\mathcal{K}$ and open $\mathcal{O}$ such that  $\mathcal{K}\subset\mathcal{O}\subset Q$, there exist a constant $C,C'>0$ so that 
$$\epsilon^2\|\Delta_\theta u\|_{L^2(\mathcal{K})}\leq C\big(\|f\|_{L^2(\mathcal{O})} +\epsilon^2 \|\nabla_\theta u\|_{L^2}+\|u\|_{L^2}\big),$$
and
$$\epsilon^{2/3}\|D^{2/3}_xu\|_{L^2(\mathcal{K})}\leq C'\big(\|f\|_{L^2(\mathcal{O})} +\epsilon^2 \|\nabla_\theta u\|_{L^2}+\|u\|_{L^2}\big).$$
Using the equation satisfied by $u$ we also deduce $\theta\cdot\nabla_xu\in L^2_{loc}(Q)$ and
$$\|\theta\cdot\nabla_xu\|_{L^2(\mathcal{K})}\leq C\big(\|f\|_{L^2(\mathcal{O})} +\epsilon^2 \|\nabla_\theta u\|_{L^2}+\|u\|_{L^2}\big),$$
for some $C$ depending on the compact $\mathcal{K}$.

\section{Explicit solution to \eqref{backward_FPB}}\label{appdx:bFPBsolution}
Let $W$ be solution to the backward Fermi pencil-beam equation
\begin{equation*}
- \tilde{\sigma}\Delta_V W  - V \cdot\nabla_{X'} W- \partial_{X^n}W  + \tilde{\lambda}W= \Psi,\quad (X,V)\in \RR^{n}_+\times\RR^{n-1},
\end{equation*}
for $\Psi$ compactly supported and with vanishing conditions at infinity
$$
\lim_{X^n\to\infty}W = 0.
$$
We solve the equation in the whole space $\RR^n\times\RR^{n-1}$ and
claim that $W$ takes the form
$$W(X,V) = \int_{X^n}^\infty e^{-\int^t_{X^n}\tilde{\lambda}(s)ds}\tau_{-X^nV}({\bf H}_3(t)*\tau_{tV}\Psi(t))dt.$$

Let's introduce the following functions:
$$U(X,V) := W(-X,V)\quad F(X,V) := \Psi(-X,V),$$
defined for $X^n\leq0$. Then $U$ satisfies the equation
$$
- \tilde{\sigma}^-\Delta_V U  + V \cdot\nabla_{X'} U+ \partial_{X^n}U  + \tilde{\lambda}^-U= F,\quad (X,V)\in \RR^{n}\times\RR^{n-1},
$$
with $\tilde{\sigma}^-(t) = \tilde{\sigma}(-t)$ and $\tilde{\lambda}^-(t) = \tilde{\lambda}(-t)$.
Let's set
$$\bfU(\xi,X^n,V) = e^{iX^n(V\cdot\xi) + \int^{X^n}_{-\infty}\tilde{\lambda}^-(s)ds}\mathfrak{F}_{X'}(U),\quad\text{and}\quad \bfF(\xi,X^n,V) = e^{iX^n(V\cdot\xi) + \int^{X^n}_{-\infty}\tilde{\lambda}^-(s)ds}\mathfrak{F}_{X'}(F),$$
therefore
$$
\begin{aligned}
\partial_{X^n}\bfU &= [i(V\cdot\xi)+\tilde{\lambda}^-]\bfU + e^{iX^n(V\cdot\xi) + \int^{X^n}_{-\infty}\tilde{\lambda}^-(s)ds}\partial_{X^n}\mathfrak{F}_{X'}(U)\\
&=\tilde{\sigma}^-e^{iX^n(V\cdot\xi) + \int^{X^n}_{-\infty}\tilde{\lambda}^-(s)ds}\Delta_V(e^{-iX^n(V\cdot\xi) - \int^{X^n}_{-\infty}\tilde{\lambda}^-(s)ds}\mathfrak{F}_{X'}(U)) + \bfF\\
&=\tilde{\sigma}^-e^{iX^n(V\cdot\xi)}\Delta_V(e^{-iX^n(V\cdot\xi)}\mathfrak{F}_{X'}(U)) + \bfF
\end{aligned}
$$
Denoting 
$$\hat{\bfU}(\xi,X^n,\eta) = \mathfrak{F}_{V}(\bfU)\quad\text{and}\quad \hat{\bfF}(\xi,X^n,\eta) = \mathfrak{F}_{V}(\bfF),$$
then taking Fourier Transform on the equation above and using twice the identity
$$\mathfrak{F}_V[e^{\ii t(V\cdot\xi)}h(V)](\eta) = \mathfrak{F}_V[h(V)](\eta-t\xi),\quad t\in\RR,$$ 
we obtain
$$
\partial_{X^n}\hat{\bfU}+\tilde{\sigma}^-|\eta-X^n\xi|^2\hat{\bfU} = \hat{\bfF}.
$$
Recalling the vanishing condition imposed at infinity, $\lim_{X^n\to-\infty}U=0$, an explicit expression for $\hat{\bfU}$ is given by
$$
\hat{\bfU}(\xi,X^n,\eta) = \int^{X^n}_{-\infty}e^{-\int^{X^n}_t|\eta - s\xi|^2\tilde{\sigma}^-(s)ds}\hat{\bfF}(\xi,t,\eta)dt.
$$
Since $\hat{\bfU}(\xi,X^n,\eta)  = e^{\int^{X^n}_{-\infty}\tilde{\lambda}^-(s)ds}\mathfrak{F}_{X',V}[U](\xi,X^n,\eta-X^n\xi)$, then
$$
\begin{aligned}
U(X',X^n,V) &=e^{-\int^{X^n}_{-\infty}\tilde{\lambda}^-(s)ds}\mathfrak{F}_{X',V}^{-1}[\hat{\bfU}(\xi,X^n,\eta+X^n\xi)](X,V)\\
&= e^{-\int^{X^n}_{-\infty}\tilde{\lambda}^-(s)ds}\mathfrak{F}_{X'}^{-1}[e^{-iX^n(V\cdot\xi)}\mathfrak{F}_{X'}^{-1}[\hat{\bfU}(\xi,X^n,\eta)]](X,V)\\
&= e^{-\int^{X^n}_{-\infty}\tilde{\lambda}^-(s)ds}\mathfrak{F}_{X',V}^{-1}[\hat{\bfU}](X'-X^nV,X^n,V).
\end{aligned}
$$
Therefore
$$
\begin{aligned}
U(X',X^n,V) &= e^{-\int^{X^n}_{-\infty}\tilde{\lambda}^-(s)ds}\tau_{X^nV}(\mathfrak{F}_{X',V}^{-1}[\hat{\bfU}])\\
&=  e^{-\int^{X^n}_{-\infty}\tilde{\lambda}^-(s)ds}\tau_{X^nV}\Big(\int^{X^n}_{-\infty}\tilde{\bH}_3(t)*\mathfrak{F}_{X',V}^{-1}[\hat{\bfF}(t)]dt\Big)\\
\end{aligned}
$$
with
$$
\tilde{\bH}_3(X',X^n,V;t) = \mathfrak{F}_{X',V}^{-1}[e^{-\int^{X^n}_t|\eta + (X^n-s)\xi|^2\tilde{\sigma}^-(s)ds}].
$$
Analogously as for $U$, we have
$$
F(X',t,V) = e^{-\int^{t}_{-\infty}\tilde{\lambda}^-(s)ds}\mathfrak{F}_{X',V}^{-1}[\hat{\bfF}](X'-tV,t,V),
$$
which implies $\mathfrak{F}_{X',V}^{-1}[\hat{\bfF}] = e^{\int^{t}_{-\infty}\tilde{\lambda}^-(s)ds}\tau_{-tV}(F)$, and consequently
$$
U(X',X^n,V) =  \int^{X^n}_{-\infty}e^{-\int^{X^n}_{t}\tilde{\lambda}^-(s)ds}\tau_{X^nV}\Big(\tilde{\bH}_3(X,V;t)*\tau_{-tV}(F(X',t,V))\Big)dt.
$$
Going back to the original function, for $X^n>0$ we obtain
$$
\begin{aligned}
W(X',X^n,V) &= U(-X',-X^n,V)\\
 &=  \int^{-X^n}_{-\infty}e^{-\int^{-X^n}_{t}\tilde{\lambda}^-(s)ds} \tau_{-X^nV}\Big(\tilde{\bH}_3(-X,V;t)*\tau_{-tV}(F(-X',t,V))\Big)dt\\
  &=  \int^{\infty}_{X^n}e^{-\int^{-X^n}_{-t}\tilde{\lambda}^-(s)ds} \tau_{-X^nV}\Big(\tilde{\bH}_3(-X,V;-t)*\tau_{tV}(F(-X',-t,V))\Big)dt\\
    &=  \int^{\infty}_{X^n}e^{-\int^{t}_{X^n}\tilde{\lambda}(s)ds} \tau_{-X^nV}\Big(\tilde{\bH}_3(-X,V;-t)*\tau_{tV}(\Psi(X',t,V))\Big)dt,
\end{aligned}
$$
We conclude by writing $\bH_3(X,V;t) = \tilde{\bH}_3(-X,V;-t)$.

\section{Proof of Lemma \ref{lemma:FPb_Lip}}\label{appdx:proof_lemma_FPb_Lip}
Let $W$ be the solution to the backward Fermi pencil-beam equation with vanishing condition at infinity  \eqref{backward_FPB}-\eqref{backward_FPB_bc}. It is given explicitly by
$$W(X,V) = \int_{X^n}^\infty e^{-\int^t_{X^n}\tilde{\lambda}(s)ds}\tau_{-X^nV}({\bf H}_3(t)*\tau_{tV}\Psi(t))dt.$$
For $(X'_1,X^n,V_1),(X'_2,X^n,V_2)\in \RR^{2(n-1)}$ we have
$$
\begin{aligned}
&|W(X'_1,X^n,V_1)-W(X'_2,X^n,V_2)| \\
&\quad \leq \int_{X^n}^\infty e^{-\int^t_{X^n}\tilde{\lambda}(s)ds}\int\int {\bf H}_3(\tilde{X}',X^n,\tilde{V};t)\\
&\hspace{2em}\times\big|\Psi(X'_1-\tilde{X}'+(X^n-t)V_1 + t\tilde{V} ,t,V_1-\tilde{V})
-\Psi(X'_2-\tilde{X}'+(X^n-t)V_2 + t\tilde{V} ,t,V_2-\tilde{V})
\big|d\tilde{X}'d\tilde{V}dt\\
&\quad \leq \int_{X^n}^\infty e^{-\int^t_{X^n}\tilde{\lambda}(s)ds}\int\int {\bf H}_3(\tilde{X}',X^n,\tilde{V};t)\\
&\hspace{2em}\times |(X'_1-X'_2+(X^n-t)(V_1-V_2),V_1-V_2)|\Lip(\Psi)d\tilde{X}'d\tilde{V}dt\\
&\quad = \Lip(\Psi)\int_{X^n}^\infty e^{-\int^t_{X^n}\tilde{\lambda}(s)ds}\underbrace{\|{\bf H}_3(X^n;t)\|_{L^1(\RR^{2(n-1)})}}_{=const.}|(X'_1-X'_2+(X^n-t)(V_1-V_2),V_1-V_2)|dt\\
&\quad \leq C\Lip(\Psi)|(X'_1,V_1)-(X'_2,V_2)|\int_{X^n}^\infty e^{-\int^t_{X^n}\tilde{\lambda}(s)ds}|t-X^n|dt\\
&\quad = C\Lip(\Psi)|(X'_1,V_1)-(X'_2,V_2)|\int_{0}^\infty e^{-\int^{\tau + X^n}_{X^n}\tilde{\lambda}(s)ds}\tau d\tau\\
&\quad \leq C\lambda_0^{-2}\Lip(\Psi)|(X'_1,V_1)-(X'_2,V_2)|.
\end{aligned}
$$
\end{appendix}

\bibliographystyle{siamplain}

\end{document}